%% file: stress_constraints.tex
\documentclass[11pt, twoside]{article}

\usepackage{authblk}
\usepackage{amsmath,amsthm}

\usepackage[utf8]{inputenc}

\usepackage{enumerate}

\usepackage{fancyhdr}
\usepackage{cite}
\usepackage{enumitem}

\usepackage{tikz}
\usetikzlibrary{patterns.meta}
\usetikzlibrary {arrows.meta}
\usetikzlibrary{decorations.pathreplacing}

\usepackage{graphics}

\usepackage{subcaption}

\usepackage{algcompatible}
\usepackage{algorithm}

\usepackage{xargs}                      
\usepackage[colorinlistoftodos,prependcaption,textsize=tiny]{todonotes}
\usepackage[normalem]{ulem}

\newcommandx{\pcomment}[2][1=]{\todo[linecolor=red,backgroundcolor=red!25,bordercolor=red,#1]{#2}}
\newcommandx{\kcomment}[2][1=]{\todo[linecolor=blue,backgroundcolor=blue!25,bordercolor=blue,#1]{#2}}

\newif\ifpreprint

\usepackage{hyperref}
\hypersetup{%
    colorlinks=True, 
    citecolor=blue,
    linkcolor=black}

\usepackage[charter]{mathdesign}

\newcommand{\rom}[1]{\uppercase\expandafter{\romannumeral #1\relax}}

  \theoremstyle{definition}
  \newtheorem{theorem}{Theorem}[section]
  \newtheorem{corollary}[theorem]{Corollary}
  
  \newtheorem{lemma}[theorem]{Lemma}
  \newtheorem{definition}[theorem]{Definition}
  \newtheorem{remark}[theorem]{Remark}

\newtheorem{thmx}{Assumption}
\newtheorem{assumption}[thmx]{Assumption}

  \newtheorem*{assumption*}{Assumption}

  \numberwithin{equation}{section}
    
\input{slidesymbols}

\setenumerate[0]{label=(\alph*)}

\newcommand{\eps}{{\varepsilon}}

\newcommand{\ben}{\begin{equation}}
\newcommand{\een}{\end{equation}}

\newcommand{\benn}{\begin{equation*}}
\newcommand{\eenn}{\end{equation*}}

\usepackage[colorinlistoftodos]{todonotes}

\usepackage{authblk}
\usepackage[T1]{fontenc}
\usepackage[utf8]{inputenc}

\begin{document}

\title{Minimisation of peak stresses with the shape derivative}
\author{Phillip Baumann and Kevin Sturm}
\affil{TU Wien}
\maketitle

\begin{abstract}
This paper is concerned with the minimisation of peak stresses occurring in linear elasticity. We propose to minimise the maximal von Mises stress of the elastic body. This leads to a nonsmooth shape functional. We derive the shape derivative and associate it with the Clarke sub-differential. Using a steepest descent algorithm we present numerical simulations. We compare our results to the usual $p$-norm regularisation and show that our algorithm performs better in the presented tests.  
\end{abstract}

\section{Introduction}\label{sec:introduction}
Mechanical stress describes forces acting on a continuous body while undergoing deformation. As mechanical stresses can lead to damage of materials in machines and structures, it is important in production processes to reduce stresses. For example stresses in iron parts of machines undergoing high stresses influence their life expectancies.  Mathematically the reduction of stresses can be expressed by so-called \emph{stress constraints} enforcing an upper limit for stresses of a material. 

First studies on (local) stress constraints can be found in \cite{a_DUBE_1998a}.   Let us mention the work \cite{a_ALJOMA_2004a} dealing with topology optimisation, where stress constraints are considered by the homogenisation method. In addition to an $L_2$ averaging of the stress, the authors of \cite{a_ALJOMA_2004a} used a weighting factor that allows to localise the stress constraints at a specific region. In \cite{a_CHNO_2010a} the authors employed the solid isotropic material with penalisation (SIMP) method to tackle stress constraints in the framework of topology optimisation. In their work, the authors discussed global as well as regional stress constraints and further introduced an iterative normalization to control local stresses. Further SIMP based approaches can be found in \cite{a_VELAKE_2017a,a_HOTOKL_2013a,a_BR_2008a,a_WAQI_2018a,a_TRTO_2018a} to name only a few. In \cite{a_KOST_2012a}, as an alternative to SIMP, a free parameter material optimisation was proposed to address the peak stresses. Furthermore, in \cite{a_BUST_2005a} the authors utilise the phase-field approach to relax local stress constraints. Besides constraints on the von Mises stress, the authors additionally consider total stress constraints and derive constraint qualifications as well as first order optimality conditions on a discrete level.

 Another approach to tackle stress constraints in structural optimisation employs gradients based on topological and shape sensitivities. In \cite{a_AMNO_2010} a topology optimisation problem dealing with linear elasticity and stress constraints is considered. In this paper the von Mises stress is regularised and the topological derivative is computed. In \cite{a_AMNO_2012a}, the authors extended this method to target constraints on the Drucker-Prager stress. Similarly, also \cite{a_SUTA_2013a} employs the topological derivative. The authors of \cite{a_ALJO_2008} compute both, the topological and shape gradient, for a cost functional similar to \cite{a_ALJOMA_2004a} with an additional volume term to regularise the problem. They used both gradients in a numerical scheme to guide a level set function. This allowed to alternately perform smooth boundary deformations as well as nucleations of holes leading to topological changes.
A related approach is utilised in \cite{a_DEST_2022}, where the authors employed a SIMP based approach to generate an initial shape followed by smooth boundary variations. In contrast to \cite{a_ALJOMA_2004a}, the authors of \cite{a_DEST_2022} did not compute the shape gradient on a continuous level but discretised the problem in terms of nodes describing the design boundary.

A common strategy to deal with stress constraints is the so-called $p$-norm approach, since the $L_p$ norm mimics pointwise maxima for sufficiently large values of $p$. This idea was employed in \cite{a_PITO_2018}, where the $p$-norm is used to simulate peak von Mises stresses. The authors first discretised the partial differential equation (PDE) and the stress constraints. Afterwards they computed shape sensitivities on a discrete level to guide a level set function. With this approach they considered the minimisation of a stress functional as well as stress constraints in the form of penalty terms. In \cite{a_PITO_2018b} the authors extended their method to further tackle stress optimisation in a given subregion of the domain.
A similar approach to treat stress constraints is followed in \cite{a_NGKI_2020}, where the authors used trimmed hexahedral meshes at the design boundary to achieve accurate stress estimations and shape sensitivities.

In this paper we are interested in the stress minimisation of a physical material described by the equations of linear elasticity.  We minimise a nonsmooth penalisation cost function that drives the stress below a certain prescribed threshold.  More specifically, we propose to directly minimise the maximum norm of the von Mises stress minus the upper allowed stress. Due to the involved maximum norm this leads naturally to a nonsmooth shape optimisation problem for which shape sensitivities can be computed. In contrast to \cite{a_PITO_2018}, we work on a continuous level and also derive shape sensitivities in the continuous setting. The resulting shape derivative at a given shape is in general nonlinear with respect to the vector field. Our strategy can be seen as an alternative to the $p$-norm approach and by nature the maximum norm captures perfectly peak stresses, which are reduced during the minimisation process. We compute the shape derivative for the maximum von Mises stress and relate the derivative to the Clarke subdifferential by using higher regularity results for linear elasticity. In the numerical part of the paper we use the shape sensitivities and a mesh deformation approach to tackle several model problems. We compare our results with the usual $p$-norm regularisation of the stress and compare the resulting stresses and designs.

\subsection*{Problem description}
We consider the following model problem. Let $\Dsf\subset\VR^d$ an open and bounded Lipschitz domain and $\Gamma\subset\Dsf$ be a smooth Lipschitz manifold. Introduce the admissible shapes by
\begin{equation}
\mathcal A:=\{\Omega\subset\Dsf|\;\Omega\text{ is Lipschitz and }\Gamma\subset\partial\Omega\}.
\end{equation}
The elastic body will be fixed at $\Gamma$. Further, our elastic body is assumed to have a target volume $V>0$ and thus we introduce the shape functional
\begin{equation}\label{eq:def_jvol}
J_\text{vol}(\Omega):=\left(|\Omega|-V\right)^2.
\end{equation}
Global minimisers of this functional are possible designs for the elastic body. Additionally, we consider a constraint on the von Mises stress
\begin{equation}
\sigma_M^2(u_\Omega)\le\delta,\quad\text{ in }\overline{\Omega\setminus\omega},
\end{equation}
 where $u_\Omega\in H^1_\Gamma(\Omega)^d$ solves the equation of linear elasticity on $\Omega$, i.e.
\begin{equation}\label{eq:state}
\int_\Omega A\eps(u_\Omega):\eps(\varphi)\;dx=\int_\Omega f\cdot\varphi\;dx+\int_{\Gamma^N}g\cdot\varphi\;dS,\quad\text{ for all }\varphi\in H^1_\Gamma(\Omega)^d.
\end{equation}
Here, $\Gamma\subset\Dsf$ denotes the portion of a Lipschitz boundary, $H^1_\Gamma(\Omega)^d:=\{u\in H^1(\Omega)^d|\;u|_\Gamma=0\}$, $\Gamma^N:=\partial\Omega\setminus\Gamma$, $\omega\subset \Dsf$ open such that $\Gamma\subset\omega$, $f\in H^1(\Dsf)^d$, $g\in H^2(\Dsf)^d$, $\delta>0$ is a given threshold and $\eps(u)$ denotes the symmetrised gradient
\begin{equation}
\eps(v):=\frac{1}{2}\left(\partial v+\partial v^\top\right),\quad\text{ for all }v\in H^1(\Dsf)^d.
\end{equation}
Furthermore for given Lam\'{e} coefficients $\lambda,\mu>0$ the elasticity tensor $A$ is defined by
\begin{equation}
A(M):=2\mu M+\lambda\text{tr}(M)I_d,\quad\text{ for all }M\in\VR^{d\times d},
\end{equation}
with $I_d$ denoting the identity matrix and the squared von Mises stress is given by
\begin{equation}
\sigma_M^2(u):=B\eps(u):\eps(u),
\end{equation}
where
\begin{equation}\label{eq:mises_tensor}
B(M):=6\mu M+(\lambda-2\mu)\text{tr}(M)I_d,\quad\text{ for all }M\in\VR^{d\times d}.
\end{equation}
We incorporate this pointwise constraint on the von Mises stress as a penalty to the objective functional $J_\text{vol}$. The stress constraints are described by minima of the following shape function
\begin{equation}\label{eq:penalty_term}
J_\sigma(\Omega):=\max\{\max_{x\in\overline{\Omega\setminus\omega}}\sigma_M^2(u_\Omega)(x)-\delta,0\},
\end{equation}
where $u_\Omega\in H^1_\Gamma(\Omega)^d$ solves \eqref{eq:state}. In order to find an elastic body $\Omega$ with minimal stresses and the target volume $V$, we introduce for a parameter $\alpha>0$ the penalised objective functional
\begin{equation}
J(\Omega):=J_\text{vol}(\Omega)+\alpha J_\sigma(\Omega).
\end{equation}
\paragraph{Structure of the paper} In Section \ref{sec:material_derivative} we compute the material derivative of the state variable and show some convergence results needed later on. In Section \ref{sec:shape_derivative} we use the direct method and a Danskin-type result, employing the results of Section \ref{sec:material_derivative}, to compute the one sided directional derivative of the cost functional. We therefore split the cost functional and deal with the smooth and nonsmooth part separately. In Section \ref{sec:hilbert_setting} we reformulate the computed derivatives in a Hilbert space setting and derive steepest descent directions  as well as optimality criteria. In Section \ref{sec:clarke} we introduce the framework of Clarke differentiation and draw a connection between our results and the Clarke subgradient. In Section \ref{sec:numerics} we show some numerical results for a simple model problem.
\section{Shape sensitivity analysis and material derivative}\label{sec:material_derivative}
Let $\Dsf\subset\VR^d$ open denote the hold-all domain, $\Omega\in\mathcal A$ be an initial shape with boundary $\partial \Omega=\Gamma\cup\Gamma^N$ and fix an open set $\omega\subset\Dsf$ such that $\Gamma\subset\omega$. In order to deduce a derivative consistent with our problem formulation in Section \ref{sec:introduction}, we consider deformation vector fields $X\in C^1(\bar\Dsf)^d$ with $\text{supp} (X)\subset (\Dsf\setminus\bar{\omega})$. On the one hand, this ensures that the portion of the boundary $\Gamma$ remains fixed. On the other hand, this choice allows us to avoid the region of low regularity, i.e. the intersection of both boundary parts.\newline
In the following we fix a vector field $X\in C^1(\bar\Dsf)^d$ with $\text{supp} (X)\subset (\Dsf\setminus\bar{\omega})$. Additionally, we denote for sufficiently small $t\ge0$ and $Y\in C^1(\bar\Dsf)^d$ with $\text{supp} (Y)\subset (\Dsf\setminus\bar{\omega})$ and $\|Y\|_{C^1(\bar\Dsf)^d}$ small the deformed domain $\Omega_{Y,t}:=F_{Y,t}(\Omega)$ and the deformed boundary $\Gamma^N_{Y,t}:=F_{Y,t}(\Gamma^N)$, where $F_{Y,t}:=\text{Id}+Y+tX$. Similarly we denote the unperturbed objects $F_Y:=F_{Y,0}$, $\Omega_Y:=\Omega_{Y,0}$ and $\Gamma^N_Y:=\Gamma^N_{Y,0}$. With the previous notations we introduce the perturbed state variable $u_{Y,t}\in H^1_\Gamma(\Omega_{Y,t})^d$ as the unique solution to
\begin{equation}\label{eq:perturbed_state}
\int_{\Omega_{Y,t}} A \eps(u_{Y,t}):\eps(\varphi)\; dx=\int_{\Omega_{Y,t}} f\cdot\varphi\; dx+\int_{\Gamma^N_{Y,t}}g\cdot\varphi\;dS,\quad\text{ for all }\varphi\in H^1_\Gamma(\Omega_{Y,t})^d,
\end{equation}
and similarly for $t=0$ we denote the unperturbed state variable $u_Y\in H^1_\Gamma(\Omega)^d$ as the unique solution to
\begin{equation}\label{eq:unperturbed_state}
\int_{\Omega_Y} A \eps(u_Y):\eps(\varphi)\; dx=\int_{\Omega_Y} f\cdot\varphi\; dx+\int_{\Gamma^N_Y}g\cdot\varphi\;dS,\quad\text{ for all }\varphi\in H^1_\Gamma(\Omega_Y)^d.
\end{equation}

    According to \cite[Theorem 2.2.2]{b_ZI_1989a} given a bi-Lipschitz transformation field $F:\VR^d\to\VR^d$, there holds \[\varphi\in H^1_\Gamma(\Omega)^d\quad\text{ iff }\quad \varphi\circ F\in H^1_\Gamma(F^{-1}(\Omega))^d,\] and thus we can reformulate the state equations \eqref{eq:perturbed_state} and \eqref{eq:unperturbed_state} on the fixed domain $\Omega$. Indeed, a change of variables shows that $u^{Y,t}:=u_{Y,t}\circ F_{Y,t}$ satisfies
\begin{equation}\label{eq:perturbed_state_lift}
\int_{\Omega} \xi_{Y,t}A \eps_{Y,t}(u^{Y,t}):\eps_{Y,t}(\varphi)\; dx=\int_{\Omega}\xi_{Y,t} f\circ F_{Y,t}\cdot\varphi\; dx+\int_{\Gamma^N}\nu_{Y,t} g\circ F_{Y,t}\cdot\varphi\;dS,\quad\text{ for all }\varphi\in H^1_\Gamma(\Omega)^d,
\end{equation}
where $\xi_{Y,t}:=\det(\partial F_{Y,t})$, $\nu_{Y,t}:=\det(\partial F_{Y,t})|(\partial F_{Y,t})^{-\top}n|$ and $\eps_{Y,t}(\varphi):=\frac{1}{2}\left(\partial \varphi (\partial F_{Y,t})^{-1}+(\partial F_{Y,t})^{-\top}\partial\varphi^\top\right)$. In this context $n$ denotes the normal vector on $\Gamma^N$. A similar computation shows
\begin{equation}\label{eq:unperturbed_state_lift}
\int_{\Omega} \xi_{Y}A \eps_{Y}(u^{Y}):\eps_{Y}(\varphi)\; dx=\int_{\Omega}\xi_{Y} f\circ F_{Y}\cdot\varphi\; dx+\int_{\Gamma^N}\nu_{Y} g\circ F_{Y}\cdot\varphi\;dS,\quad\text{ for all }\varphi\in H^1_\Gamma(\Omega)^d,
\end{equation}
where $\xi_{Y}:=\det(\partial F_{Y})$, $\nu_{Y}:=\det(\partial F_{Y})|(\partial F_{Y})^{-\top}n|$ and $\eps_{Y}(\varphi):=\frac{1}{2}\left(\partial \varphi (\partial F_{Y})^{-1}+(\partial F_{Y})^{-\top}\partial\varphi^\top\right)$.
In the following lemma and also subsequently throughout the paper, the notation $Y\to 0$ is understood with respect to the norm on $C^1(\bar\Dsf)^d$, i.e. $\|Y\|_{C^1(\bar\Dsf)^d}\to 0$.
\begin{lemma}\label{lem:transformations}
Let $\Omega\subset\Dsf$ a Lipschitz set. Furthermore, let $f\in H^1(\Dsf)^d$, $g\in H^2(\Dsf)^d$, $X\in C^1(\bar\Dsf)^d$ with $\text{supp} (X)\subset (\Dsf\setminus\bar{\omega})$ and define for $t>0$ sufficiently small and $Y\in C^1(\bar\Dsf)^d$ with $\text{supp} (Y)\subset (\Dsf\setminus\bar{\omega})$ and $\|Y\|_{C^1(\bar\Dsf)^d}$ sufficiently small. Then there holds
\begin{itemize}
\item[(i)] $\displaystyle\lim_{\substack{Y\to 0\\ t\searrow 0}}\left\|\frac{f\circ F_{Y,t}-f\circ F_Y}{t}-\partial fX\right\|_{L_2(\Dsf)^d}=0$,
\item[(ii)] $\displaystyle\lim_{\substack{Y\to 0\\ t\searrow 0}}\left\|\frac{g\circ F_{Y,t}-g\circ F_Y}{t}-\partial gX\right\|_{L_2(\partial\Omega)^d}=0$.
\end{itemize}
\end{lemma}

\begin{proof}
ad (i): First, note that for $\varphi\in C^\infty(\Dsf)^d$ there holds
\begin{equation}\label{eq:taylor_approx}
\varphi\circ F_{Y,t}(x)-\varphi\circ F_Y(x)=t\int_0^1\partial \varphi(x+Y(x)+stX(x))X(x)\;ds.
\end{equation}
Since the integrand is bounded, we observe $|\varphi\circ F_{Y,t}(x)-\varphi\circ F_Y(x)|\le Ct$ and thus further
\begin{equation}\label{eq:approx_f_smooth}
\|\varphi^{Y,t}-\varphi^Y\|_{L_2(\Dsf)^d}\le Ct,
\end{equation}
where here and throughout the proof $C\in\VR$ denotes a constant independent of $t$ and $Y$. Next we consider $f\in H^1(\Dsf)^d$ and choose a smooth approximation. That is, for each $\eps>0$ a function $f_\eps\in C^\infty(\Dsf)^d$ such that $\|f_\eps-f\|_{H^1(\Dsf)^d}\le \eps$. An application of the triangle inequality yields
\begin{equation}\label{eq:approx_f_tri}
\|f^{Y,t}-f^Y\|_{L_2(\Dsf)^d}\le \|f^{Y,t}-f_\eps^{Y,t}\|_{L_2(\Dsf)^d}+\|f_\eps^{Y,t}-f_\eps^Y\|_{L_2(\Dsf)^d}+\|f_\eps^{Y}-f^Y\|_{L_2(\Dsf)^d}.
\end{equation}
A change of variables now shows that
\begin{equation}
\|f^{Y,t}-f_\eps^{Y,t}\|^2_{L_2(\Dsf)^d}=\int_\Dsf|f^{Y,t}-f_\eps^{Y,t}|^2\;dx=\int_\Dsf\Det(I+\partial Y+t\partial X)^{-1}|f-f_\eps|^2\;dx.
\end{equation}
Since $\Det(I+\partial Y+t\partial X)\to1$ as $Y\to0$ in $C^1(\bar\Dsf)^d$ and $t\searrow 0$, there holds $\Det(I+\partial Y+t\partial X)^{-1}\le C$ for $t>0$ and $\|Y\|_{C^1(\bar\Dsf)^d}$ sufficiently small. The same argument can be applied to the third term on the right hand side of \eqref{eq:approx_f_tri}. Thus, choosing $\eps=t$ and inserting \eqref{eq:approx_f_smooth} with $\varphi=f_\eps$ into \eqref{eq:approx_f_tri}, we conclude
\begin{equation}
\|f^{Y,t}-f^Y\|_{L_2(\Dsf)^d}\le Ct,
\end{equation}
which shows (i).\newline
ad (ii): We start again with a smooth function $\varphi\in C^\infty(\Dsf)^d$. Using Lebesgue's dominated convergence theorem, we deduce from equation \eqref{eq:taylor_approx} that
\begin{equation}\label{eq:pointwise_convergence}
\frac{\varphi^{Y,t}(x)-\varphi^Y(x)}{t}\to \partial \varphi(x)X(x)\quad\text{ for all }x\in\Dsf.
\end{equation}
Since $X$ has compact support, there further holds $|\frac{\varphi^{Y,t}(x)-\varphi^Y(x)}{t}|\le C$ for all $x\in\Dsf$. Now, as $\Dsf$ has finite measure, we can employ dominated convergence to extend the convergence \eqref{eq:pointwise_convergence} to $L_2$. That is, 
\begin{equation}\label{eq:approx_f_H1}
\|\frac{\varphi^{Y,t}-\varphi^Y}{t}- \partial \varphi X\|_{L_2(\Dsf)^d}\to 0.
\end{equation}
For $f\in H^1(\Dsf)^d$ we use a smooth approximation, i.e. for each $\eps>0$ let $f_\eps\in C^\infty(\Dsf)^d$ such that \[\|f_\eps-f\|_{H^1(\Dsf)^d}\le \eps.\] Now a splitting similar to the previous proof entails
\begin{equation}
\begin{split}
    \|\frac{f^{Y,t}-f^Y}{t}-\partial fX\|_{L_2(\Dsf)^d}\le&  \frac{1}{t}\|f^{Y,t}-f_\eps^{Y,t}\|_{L_2(\Dsf)^d}+\frac{1}{t}\|f^{Y}-f_\eps^{Y}\|_{L_2(\Dsf)^d}\\
                                                          & +\|(f_\eps^{Y,t}-f_\eps^Y)/t-\partial f_\eps X\|_{L_2(\Dsf)^d}+\|\partial f_\eps X-\partial f X\|_{L_2(\Dsf)^d}.
\end{split}
\end{equation}
Now choosing $\eps=t^2$ and using the same arguments as in the proof of item (i) to estimate the first second and fourth term gives
\begin{equation}
\|\frac{f^{Y,t}-f^Y}{t}-\partial fX\|_{L_2(\Dsf)^d}\le Ct+\|\frac{f_\eps^{Y,t}-f_\eps^Y}{t}-\partial f_\eps X\|_{L_2(\Dsf)^d}.
\end{equation}
Finally equation \eqref{eq:approx_f_H1} applied to $\varphi=f_\eps$ shows (ii).\newline
The proof of item (iii) and (iv) follows from similar arguments, using the trace inequality $\|\varphi\|_{L_2(\partial \Omega)^d}\le C\|\varphi\|_{H^1(\Dsf)^d}$, for all $\varphi\in H^1(\Dsf)^d$. The higher regularity is used to approximate $\partial g$ in the $H^1$ norm. 
\end{proof}
\begin{lemma}\label{lem:derivatives}
Let $\Omega\subset\Dsf$ Lipschitz and $X \in C^1(\bar\Dsf)^d$ with $\text{supp} (X)\subset (\Dsf\setminus\bar{\omega})$. Furthermore, for $Y\in C^1(\bar\Dsf)^d$ with $\text{supp} (Y)\subset (\Dsf\setminus\bar{\omega})$ and $\|Y\|_{C^1(\bar\Dsf)^d}$, $t>0$ sufficient small recall the notations $\xi_{Y,t}:=\det(\partial F_{Y,t})$, $\xi_{Y}:=\det(\partial F_{Y})$, $\nu_{Y,t}:=\det(\partial F_{Y,t})|(\partial F_{Y,t})^{-\top}n|$, $\nu_{Y}:=\det(\partial F_{Y})|(\partial F_{Y})^{-\top}n|$ and additionally let $\alpha_{Y,t}:=(\partial F_{Y,t})^{-1}$, $\alpha_{Y}:=(\partial F_{Y})^{-1}$. Then the following limits hold.
\begin{itemize}
\item[(i)] $\displaystyle\lim_{\substack{Y\to 0\\ t\searrow 0}}\frac{\xi_{Y,t}-\xi_Y}{t} = \Div(X)$ in $C^0(\Dsf)$,
\item[(ii)] $\displaystyle\lim_{\substack{Y\to 0\\ t\searrow 0}}\frac{\alpha_{Y,t}-\alpha_Y}{t}=-\partial X$ in $C^0(\Dsf)^{d\times d}$,
\item[(iii)] $\displaystyle\lim_{\substack{Y\to 0\\ t\searrow 0}}\frac{\nu_{Y,t}-\nu_Y}{t}=\Div X-(\partial Xn)\cdot n$ in $C^0(\partial \Omega)$.
\end{itemize}
\end{lemma}
\begin{proof}
ad (i): Using an asymptotic expansion of $\Det(I+\partial Y+t\partial X)$ with respect to the parameter $t$ yields
\begin{equation}
\Det(I+\partial Y+t\partial X)=\Det(I+\partial Y)+t\Det(I+\partial Y)\text{tr}((I+\partial Y)^{-1}\partial X)+O(t^2;\partial X;\partial Y),
\end{equation}
where the remainder satisfies $\displaystyle\lim_{\substack{Y\to 0\\ t\searrow 0}}\frac{O(t^2;\partial X;\partial Y)}{t}\to 0$ uniformly in $x$.
Hence, we conclude that
\begin{equation}
\lim_{\substack{Y\to 0\\ t\searrow 0}}\frac{\xi_{Y,t}-\xi_Y}{t}=\lim_{\substack{Y\to 0\\ t\searrow 0}}\Det(I+\partial Y)\text{tr}((I+\partial Y)^{-1}\partial X)+\frac{O(t^2;\partial X;\partial Y)}{t}=\text{tr}(\partial X)=\Div(X)
\end{equation}
holds uniformly in $\Dsf$, which shows (i).\newline
ad (ii): First, we expand the terms in a Neumann series: for sufficiently small $\|Y\|_{C^1(\bar\Dsf)^d}$ and $t>0$ there holds
\begin{equation}
(I+\partial Y)^{-1}=\sum_{\ell=0}^\infty (-1)^\ell(\partial Y)^\ell,
\end{equation}
and
\begin{equation}
(I+\partial Y+t\partial X)^{-1}=\sum_{\ell=0}^\infty (-1)^\ell(\partial Y+t\partial X)^\ell=\sum_{\ell=0}^\infty (-1)^\ell\sum_{k=0}^\ell(\partial Y)^kt^{\ell-k}(\partial X)^{\ell-k}.
\end{equation}
Hence, we observe
\begin{align*}
\frac{\alpha_{Y,t}-\alpha_Y}{t}=&\sum_{\ell=0}^\infty (-1)^\ell\sum_{k=0}^{\ell-1}(\partial Y)^kt^{\ell-k-1}(\partial X)^{\ell-k}\\
=&(-\partial X)+\sum_{\ell=2}^\infty (-1)^\ell\sum_{k=0}^{\ell-2}(\partial Y)^kt^{\ell-k-1}(\partial X)^{\ell-k}+\sum_{\ell=2}^\infty (-1)^\ell(\partial Y)^{\ell-1}\\
=&(-\partial X)+t\sum_{\ell=2}^\infty (-1)^\ell\sum_{k=0}^{\ell-2}(\partial Y)^kt^{\ell-k-2}(\partial X)^{\ell-k}+(\partial Y)\sum_{\ell=0}^\infty (-1)^\ell(\partial Y)^{\ell}.
\end{align*}
Since both series on the right hand side converge uniformly in $\Dsf$, passing to the limit $Y\to0$ in $C^1(\bar\Dsf)^d$, $t\searrow 0$ shows
\begin{equation}
\lim_{\substack{Y\to 0\\ t\searrow 0}}\frac{\alpha_{Y,t}-\alpha_Y}{t}=-\partial X\quad\text{ in }C^0(\Dsf)^{d\times d},
\end{equation}
which shows (ii).\newline
ad (iii): Expanding the difference yields
\begin{equation}\label{eq:diff_boundaryform}
\frac{\nu_{Y,t}-\nu_Y}{t}=\frac{\Det(\partial F_{Y,t})-\Det(\partial F_Y)}{t}|\partial F_{Y,t}n|+\Det(\partial F_Y)\frac{|\partial F_{Y,t}n|-|\partial F_{Y}n|}{t}.
\end{equation}
Using item (i) and $|n|=1$, passing to the limit in the first term on the right hand side of \eqref{eq:diff_boundaryform} shows
\begin{equation}
\displaystyle\lim_{\substack{Y\to 0\\ t\searrow 0}}\frac{\Det(\partial F_{Y,t})-\Det(\partial F_Y)}{t}|\partial F_{Y,t}n|=\Div(X),
\end{equation}
uniformly in $x$.
Regarding the limit of the remaining term we observe
\begin{align*}
\Det(\partial F_Y)\frac{|\partial F_{Y,t}n|-|\partial F_{Y}n|}{t}=&\Det(\partial F_Y)\frac{|\partial F_{Y,t}n|^2-|\partial F_{Y}n|^2}{t}\left(|\partial F_{Y,t}n|+|\partial F_{Y}n|\right)^{-1}\\
=&\Det(\partial F_Y)\frac{(\partial F_{Y,t}n-\partial F_{Y}n)\cdot (\partial F_{Y,t}n+\partial F_{Y}n)}{t}\left(|\partial F_{Y,t}n|+|\partial F_{Y}n|\right)^{-1}\\
=&\Det(\partial F_Y)\left(\frac{\partial F_{Y,t}-\partial F_{Y}}{t}n\right)\cdot (\partial F_{Y,t}n+\partial F_{Y}n)\left(|\partial F_{Y,t}n|+|\partial F_{Y}n|\right)^{-1}.
\end{align*}
Now item (ii) yields the uniform limit
\begin{equation}
\displaystyle\lim_{\substack{Y\to 0\\ t\searrow 0}}\Det(\partial F_Y)\frac{|\partial F_{Y,t}n|-|\partial F_{Y}n|}{t}=\left(-\partial Xn\right)\cdot n.
\end{equation}
\end{proof}
We are now able to state the main result of this section, which covers the material derivative of the state equation.
\begin{lemma}\label{lem:material_bound}
Let $\Omega\subset\Dsf$ Lipschitz and $X \in C^1(\bar\Dsf)^d$ with $\text{supp} (X)\subset (\Dsf\setminus\bar{\omega})$. Furthermore, for $Y\in C^1(\bar\Dsf)^d$ with $\text{supp} (Y)\subset (\Dsf\setminus\bar{\omega})$ and $\|Y\|_{C^1(\bar\Dsf)^d}$, $t>0$ sufficient small let $u^{Y,t}\in H^1_\Gamma(\Omega)^d$ be defined in \eqref{eq:perturbed_state_lift}. Then there is a constant $C\in\VR$ independent of $t$ and $Y$ such that
\begin{equation*}
\|u^{Y,t}\|_{H^1(\Omega)^d}\le C.
\end{equation*}
\end{lemma}
\begin{proof}
Testing \eqref{eq:perturbed_state_lift} with $\varphi=u^{Y,t}$ yields
\[\int_{\Omega} \xi_{Y,t}A \eps_{Y,t}(u^{Y,t}):\eps_{Y,t}(u^{Y,t})\; dx=\int_{\Omega}\xi_{Y,t} f^{Y,t}\cdot u^{Y,t}\; dx+\int_{\Gamma^N}\nu_{Y,t} g^{Y,t}\cdot u^{Y,t}\;dS.\]
Thus, due to the differentiability results of Lemma \ref{lem:transformations}, Lemma \ref{lem:derivatives} and Hölder's inequality, we deduce
\[\|\eps(u^{Y,t})\|_{L_2(\Omega)^d}^2\le C \|u^{Y,t}\|_{H^1(\Omega)^d}.\]
Now the result follows from Korn's inequality.
\end{proof}
\begin{lemma}\label{lem:material_diff_bound}
With the assumptions of Lemma \ref{lem:material_bound} let $u^{Y,t}$ and $u^Y$ be defined in \eqref{eq:perturbed_state_lift} and \eqref{eq:unperturbed_state_lift}, respectively. Then there is a constant $C\in\VR$ independent of $Y$ and $t$ such that
\begin{equation}\label{eq:aux_lemma2_material}
\|u^{Y,t}-u^Y\|_{H^{1}(\Omega)^d}\le Ct.
\end{equation}
\end{lemma}
\begin{proof}
Subtracting equations \eqref{eq:perturbed_state_lift} and \eqref{eq:unperturbed_state_lift} yields
\begin{align}\label{eq:diff_states}
\begin{split}
\int_\Omega \xi_YA \eps_Y(u^{Y,t}-u^Y):\eps_Y(\varphi)\; dx
=&\int_\Omega \xi_{Y,t} A [\eps_Y(u^{Y,t})-\eps_{Y,t}(u^{Y,t})]:\eps_Y(\varphi)\; dx\\
&+\int_\Omega [\xi_Y-\xi_{Y,t}]A\eps_Y(u^{Y,t}):\eps_Y(\varphi)\; dx\\
&+\int_\Omega \xi_{Y,t}A \eps_{Y,t}(u^{Y,t}):[\eps_Y(\varphi)-\eps_{Y,t}(\varphi)]\; dx\\
&+\int_\Omega (\xi_{Y,t}-\xi_Y)f^{Y,t}\cdot\varphi\; dx+\int_\Omega \xi_Y(f^{Y,t}-f^Y)\cdot\varphi\; dx\\
&+\int_{\Gamma^N} (\nu_{Y,t}-\nu_Y)g^{Y,t}\cdot\varphi\; dS+\int_{\Gamma^N} \nu_Y(g^{Y,t}-g^Y)\cdot\varphi\; dS,
\end{split}
\end{align}
for all $\varphi\in H^1_\Gamma(\Omega)^d$. Now testing with $\varphi=u^{Y,t}-u^Y$ and using similar arguments as in the proof of Lemma \ref{lem:material_bound} shows \eqref{eq:aux_lemma2_material}.
\end{proof}
\begin{theorem}\label{thm:material_derivative}
With the assumptions of the previous lemma there holds
\begin{equation}\label{eq:material_strong}
\lim_{\substack{Y\to 0\\ t\searrow 0}}\|\frac{u^{Y,t}-u^Y}{t}-\dot u\|_{H^1(\Omega)^d}=0,
\end{equation}
where $\dot u\in H^1_\Gamma(\Omega)$ is the unique solution to
\begin{align}\label{eq:u_dot}
\begin{split}
\int_\Omega A \eps(\dot u):\eps(\varphi)\; dx=&\frac{1}{2}\int_\Omega A [\partial u\partial X+\partial X^\top\partial u^\top]:\eps(\varphi)\; dx\\
&-\int_\Omega \Div XA\eps(u):\eps(\varphi)\; dx\\
&+\frac{1}{2}\int_\Omega A \eps(u):[\partial \varphi\partial X+\partial X^\top\partial \varphi^\top]\; dx\\
&+\int_\Omega \Div X f\cdot \varphi\; dx+\int_\Omega \partial f X\cdot\varphi\; dx\\
&+\int_{\Gamma^N} (\Div X-(\partial Xn)\cdot n) g\cdot \varphi\; dS+\int_{\Gamma^N} \partial g X\cdot\varphi\; dS,
\end{split}
\end{align}
for all $\varphi\in H^1_\Gamma(\Omega)^d$. 
\end{theorem}
\begin{proof}
First, consider sequences $t_n\searrow 0$ and $Y_n\in C^1(\bar\Dsf)^d$ with $\text{supp} (Y_n)\subset (\Dsf\setminus\bar{\omega})$, $Y_n\to0$ in $C^1(\bar\Dsf)^d$ starting sufficiently close to $0$. From Lemma \ref{lem:material_diff_bound} we know that $V^{Y_n,t_n}:=\frac{u^{Y_n,t_n}-u^{Y_n}}{t_n}$ is bounded in $H^1(\Omega)^d$ Hence, there exists $V\in H^1_\Gamma(\Omega)^d$ such that, up to a subsequence denoted the same, $V^{Y_n,t_n}\rightharpoonup V$ in $H^1(\Omega)^d$. Now dividing \eqref{eq:diff_states} by $t_n$, using an arbitrary testfunction $\varphi\in H^1_\Gamma(\Omega)^d$ and passing to the limit according to Lemma \ref{lem:derivatives} and Lemma \ref{lem:transformations}, we observe that $V$ satisfies \eqref{eq:u_dot}. By uniqueness we conclude that $V = \dot u$ and thus also that $\frac{u^{Y,t}-u^{Y}}{t}\rightharpoonup \dot u$ in $H^1(\Omega)^d$. Furthermore, subtracting \eqref{eq:diff_states} divided by $t$ and \eqref{eq:u_dot} allows to deduce the strong convergence \eqref{eq:material_strong} with the same arguments used in the proof of Lemma \ref{lem:material_diff_bound}.
\end{proof}
\begin{assumption}\label{ass:material_c1}
There hold the stronger results $u^{Y,t}\in C^1(K)^d$,
\begin{equation}
\lim_{\substack{Y\to 0\\ t\searrow 0}}\|u^{Y,t}-u\|_{C^1(K)^d}=0,
\end{equation}
as well as
\begin{equation}
\lim_{\substack{Y\to 0\\ t\searrow 0}}\|\frac{u^{Y,t}-u^Y}{t}-\dot u\|_{C^1(K)^d}=0,
\end{equation}
where $K:=\overline{\Omega\setminus\omega}$.
\end{assumption}
\begin{remark}
We would like to point out that Assumption \ref{ass:material_c1} is of reasonable nature. The results obtained in \cite{b_CI_2022} and further literature mentioned therein ensure that a similar analysis can be carried out to deduce convergence in $W^{2,p}(\Omega)^d$, $p>d$. However, this only holds true for pure traction or pure Dirichlet problems with sufficient regular data. Unfortunately, it is known that mixed boundary value problems lack regularity in the vicinity of the intersection $\Gamma\cap\Gamma^N$, i.e. the region where different boundary conditions collide. Thus, a similar estimate is not possible for our problem under consideration. Nonetheless, our assumption compensates this lack of regularity by the introduction of $\omega$, which ensures that the problematic region is not considered.
\end{remark}
\section{Shape derivative}\label{sec:shape_derivative}
In this section we are going to compute the first order shape derivative of the penalised objective functional. Due to the difference in complexity, we are going to treat the cost functional and the penalty term separately.\newline
Given a deformation vector field $X \in C^1(\bar\Dsf)^d$ with $\text{supp} (X)\subset (\Dsf\setminus\bar{\omega})$ and $t>0$ sufficiently small, we denote the perturbed identity $F_t:=\text{Id}+tX$ and let $\Omega_t:=F_t(\Omega)$. We refer to \cite{b_DEZO_2011a,b_HEPI_2005a,b_SOZO_1992a} for basic results and definitions regarding the shape derivative. Let us note that in the literature the term \emph{shape derivative} is not used uniformly. Usually the shape derivative is required to be linear with respect to the vector field $X$, which is in our problem not the case. However, to simplify the terminology we will use the term \emph{shape derivative} in the sense of the following definition. 
\begin{definition}
With the previous notation we define the first order shape derivative of the functional $J:\mathcal A\to\VR$ in direction $X \in C^1(\bar\Dsf)^d$ with $\text{supp} (X)\subset (\Dsf\setminus\bar{\omega})$ as
\begin{equation}
D J(\Omega)(X):=\lim_{t\searrow 0}\frac{J(\Omega_t)- J(\Omega)}{t}.
\end{equation}
\end{definition}
For the sake of simplicity, in the following we denote for $t>0$ small the perturbed and unperturbed state variables $u_t:=u_{0,t}$ and $u:=u_0$, where $u_{t,0}$ and $u_0$ are defined by \eqref{eq:perturbed_state} and \eqref{eq:unperturbed_state}, respectively. Similarly, we denote $u^t:=u_t\circ F_t$.\newline
The next lemma covers the derivative of the smooth cost functional.
\begin{lemma}\label{thm:shape_1}
Let $ J_\text{vol}:\mathcal A\to\VR$ be defined in \eqref{eq:def_jvol} and $X \in C^1(\bar\Dsf)^d$ with $\text{supp} (X)\subset (\Dsf\setminus\bar{\omega})$. Then there holds
\begin{equation}
D J_\text{vol}(\Omega)(X)=2\left(|\Omega|-V\right)\int_\Omega\Div(X)\;dx.
\end{equation}
\end{lemma}
\begin{proof}
    A change of variables shows $|\Omega_t|=\int_\Omega \Det(\text{Id}+t\partial X)\;dx $
and thus the result follows from Lemma~\ref{lem:derivatives} item (i) with $Y\equiv 0$.
\end{proof}
In order to derive a similar result for the penalty term, we need to following Danskin type result.
\begin{lemma}\label{thm:danskin}
Let $K\subset \VR^d$ compact, $\tau>0$ and $g:[0,\tau]\times K\to\VR$ some function. Additionally, define for $t\in[0,\tau]$ the set $R^t:=\{z\in K|\;\max_{x\in K}g(t,x)=g(t,z)\}$ with the convention $R:=R^0$. Further assume that
\begin{itemize}
\item[(A1)] for all $x\in R$ the partial derivative $\partial_t g(0^+,x)$ exists,
\item[(A2)] for all $t\in[0,\tau]$ the function $x\mapsto g(t,x)$ is upper semicontinuous,
\item[(A3)] for all real nullsequences $(t_n)$, $t_n\searrow 0$ and all sequences $(y_{t_n})$ converging to some $y\in R$ we have
\begin{equation}
\lim_{n\to\infty}\frac{g(t_n,y_{t_n})-g(0,y_{t_n})}{t_n}=\partial_t g(0^+,y).
\end{equation}
\end{itemize}
Then
\begin{equation}
\frac{\partial}{\partial t}\left(\max_{x\in K}g(t,x)\right)_{t=0}=\max_{x\in R}\partial_t g(0^+,x).
\end{equation}
\end{lemma}
\begin{proof}
For a proof we refer to \cite[Lemma 2.19]{a_ST_2016a}.
\end{proof}
Furthermore, we can treat the interior smooth part of the penalty term similar to Lemma \ref{thm:shape_1}.
\begin{lemma}\label{lem:shape_inner}
For $x\in K$ and $t\ge 0$ let $J_\sigma^x(\Omega_t):=\sigma_M^2(u_t)(x_t)-\delta$, where $x_t:=F_t(x)$ and $\delta\in\VR$ denotes the stress threshold. Additionally, let $X \in C^1(\bar\Dsf)^d$ with $\text{supp} (X)\subset (\Dsf\setminus\bar{\omega})$. Then there holds
\begin{equation}\label{eq:shape_stress_inner}
D J_\sigma^x(\Omega)(X)=2 B\eps(u)(x):\eps(\dot u)(x)-B[\partial u\partial X+(\partial X)^\top (\partial u)^\top](x):\eps(u)(x),
\end{equation}
with the constant tensor $B$ defined in \eqref{eq:mises_tensor} and $\dot u$ solving \eqref{eq:u_dot}.
\end{lemma}
\begin{proof}
We first perform a change of variables to get
\begin{equation}
 J_\sigma^x(\Omega_t)=B\eps_t(u^t)(x):\eps_t(u^t)(x)-\delta. 
\end{equation}
In view of Assumption~\ref{ass:material_c1} we may differentiate pointwise with respect to $t$ and obtain:
\begin{align}
D J_\sigma^x(\Omega)(X)=&-\frac{1}{2}B[\partial u\partial X+(\partial X)^\top (\partial u)^\top]:\eps(u)(x)+B\eps(\dot u)(x):\eps(u)(x)\\
&-B\eps(u)(x):\frac{1}{2}[\partial u\partial X+(\partial X)^\top (\partial u)^\top]+B\eps(u)(x):\eps(\dot u)(x).
\end{align}
By symmetry of the inner product we conclude \eqref{eq:shape_stress_inner}.
\end{proof}
Now we are able to prove the main result regarding the derivative of the penalty term. We will do this in two steps.
\begin{lemma}\label{lem:shape_2}
Let $\tilde{ J_\sigma}:\mathcal A\to\VR$ be defined by
\begin{equation}
\tilde{ J_\sigma}(\Omega):=\max_{x\in K}J_\sigma^x(\Omega)
\end{equation}
and $X \in C^1(\bar\Dsf)^d$ with $\text{supp} (X)\subset (\Dsf\setminus\bar{\omega})$. Then there holds
\begin{equation}
D\tilde{J_\sigma}(\Omega)(X)=\max_{x\in A(u)}2 B\eps(u)(x):\eps(\dot u)(x)-B[\partial u\partial X+(\partial X)^\top (\partial u)^\top](x):\eps(u)(x),
\end{equation}
where $A(u):=\{z\in K|\; \sigma_M^2(u)(z)=\max_{x\in K}\sigma_M^2(u)(x)\}$.
\end{lemma}
\begin{proof}
First note that we can rewrite the functional as
\begin{equation}
\tilde{ J_\sigma}(\Omega_t)=\max_{x\in K} J_\sigma^x(\Omega_t).
\end{equation}
Now our goal is to apply Theorem \ref{thm:danskin} to
\begin{equation}
g:[0,\tau]\times K\to \VR, \quad (t,x)\mapsto  J_\sigma^x(\Omega_t),
\end{equation}
where $\tau>0$ is a sufficiently small constant. In order to apply the theorem, we need to check the assumptions (A1)-(A3). Since $R=A(u)\subset K$, Lemma \ref{lem:shape_inner} show that (A1) is satisfied. To show (A2), note that, due to the Sobolev embedding, $u^t\in C^1(K)^d$ and thus the mapping $x\mapsto g(t,x)$ is continuous for $t\in[0,\tau]$. The proof of assumption (A3) follows the lines of Lemma \ref{lem:shape_inner}, where we further use the uniform convergence $\frac{u^t-u}{t}\to\dot u$ in $C^1$.
\end{proof}
\begin{theorem}\label{thm:shape_2}
Let $ J_\sigma:\mathcal A\to \VR$ be defined  as in \eqref{eq:penalty_term} and $X \in C^1(\bar\Dsf)^d$ with $\text{supp} (X)\subset (\Dsf\setminus\bar{\omega})$. Then there holds:
\begin{equation}
D J_\sigma(\Omega)(X)=\begin{cases}\max_{x\in A(u)} f(x),&\quad\text{ if }\sigma_M^2(u)|_{A(u)}>\delta,\\
\max\{\max_{x\in A(u)}f(x),0\},&\quad\text{ if }\sigma_M^2(u)|_{A(u)}=\delta,\\
0,&\quad\text{ if }\sigma_M^2(u)|_{A(u)}<\delta,
\end{cases}
\end{equation}
where $A(u):=\{z\in K|\; \sigma_M^2(u)(z)=\max_{x\in K}\sigma_M^2(u)(x)\}$ and $f(x):= 2 B\eps(u)(x):\eps(\dot u)(x)-B[\partial u\partial X+(\partial X)^\top (\partial u)^\top](x):\eps(u)(x)$.
\end{theorem}
\begin{proof}
The proof is another application of Theorem~\ref{thm:danskin} where this time the compact set $K$ denotes a two valued index set. I.e. $K:=\{1,2\}$ and
\begin{equation}
g(t,x):=\begin{cases}\tilde{ J_\sigma}(\Omega_t),&\quad\text{ if }x=1,\\
0,&\quad\text{ if }x=2.\end{cases}
\end{equation}
Now the result follows similarly to the previous lemma, where we note that the active set $R$ of the two valued set correlates to the three cases as follows
\begin{equation}
\begin{cases}
R=\{1\},&\quad\text{ iff }\sigma_M^2(u)|_{A(u)}>\delta,\\
R=\{1,2\},&\quad\text{ iff }\sigma_M^2(u)|_{A(u)}=\delta,\\
R=\{2\},&\quad\text{ iff }\sigma_M^2(u)|_{A(u)}<\delta.
\end{cases}
\end{equation}
\end{proof}
\begin{remark}\label{rem:adjoint_stress}
For the numerical implementation it is convenient to have access to the adjoint states corresponding to the smooth interior functionals $J_\sigma^x$ for $x\in K$. In this case, for a given $x\in K$ the adjoint state $q^x$ should satisfy the equation
\begin{equation}\label{eq:adjoint_direct}
\int_\Omega A\eps(\varphi):\eps(q^x)\;dx=2B\eps(u)(x):\eps(\varphi)(x),\quad\text{ for all }\varphi\in H^1_\Gamma(\Dsf)^d.
\end{equation}
Unfortunately, equation \eqref{eq:adjoint_direct} is not well defined, since point evaluation of the gradient is not possible in $H^1$. As a consequence, we introduce for given $r>0$ the approximation $q^{x,r}\in H^1_\Gamma(\Dsf)^d$
\begin{equation}\label{eq:adjoint_approx}
\int_\Omega A\eps(\varphi):\eps(q^{x,r})\;dx=\frac{2}{|B_r(x)|}\int_{B_r(x)}B\eps(u):\eps(\varphi)\;dx,\quad\text{ for all }\varphi\in H^1_\Gamma(\Dsf)^d.
\end{equation}
Here, $B_r(x)$ denotes the ball with center $x$ and radius $r$.
\end{remark}
\section{Hilbert space setting}\label{sec:hilbert_setting}
In this section we are going to derive optimality conditions and steepest descent directions for the penalised objective functional $ J_\text{vol}+\alpha J_\sigma$. In order to simplify our notation, we are going to reformulate our previous results in a Hilbert space setting. Therefore, let $\mathcal H$ a Hilbert space such that $\mathcal H\subset \{X \in C^1(\bar\Dsf)^d|\; \text{supp} (X)\subset (\Dsf\setminus\bar{\omega})\}$ and point evaluation of the gradient as well as the mapping $X\mapsto D J_\text{vol}(\Omega)(X)$ is continuous. \newline
Since, for $x\in A(u)$ the mappings $X\mapsto D J_\text{vol}(\Omega)(X)$ and $X\mapsto D J_\sigma^x(\Omega)(X)$ are linear, and by our assumptions also continuous, the Riesz representation theorem states that there are elements $\nabla J_\text{vol}\in \mathcal H$ and $\nabla J_\sigma^x\in\mathcal H$, such that
\begin{equation}
D J_\text{vol}(\Omega)(X)=\langle \nabla J_\text{vol},X\rangle_{\mathcal H}\qquad D J_\sigma^x(\Omega)(X)=\langle \nabla J_\sigma^x,X\rangle_{\mathcal H},
\end{equation}
for all $x\in A(u)$ and $X\in \mathcal H$. Using this representation, we are able to rewrite the shape derivative of the penalised functional $J$ as follows
\begin{equation}
DJ(\Omega)(X)=\begin{cases}
\max_{x\in A(u)}\langle \nabla J_\text{vol}+\alpha \nabla J_\sigma^x,X\rangle_{\mathcal H},&\quad\text{ if }\sigma_M^2(u)|_{A(u)}>\delta,\\
\max\{\max_{x\in A(u)}\langle \nabla J_\text{vol}+\alpha \nabla J_\sigma^x,X\rangle_{\mathcal H},\langle \nabla J_\text{vol},X\rangle_{\mathcal H}\},&\quad\text{ if }\sigma_M^2(u)|_{A(u)}=\delta,\\
\langle \nabla J_\text{vol},X\rangle_{\mathcal H},&\quad\text{ if }\sigma_M^2(u)|_{A(u)}<\delta.
\end{cases}
\end{equation}
With the definition
\begin{equation}\label{eq:L_cases}
\mathcal Z:=\begin{cases}
\{\nabla J_\text{vol}+\alpha\nabla J_\sigma^x|\;x\in A(u)\},&\quad\text{ if }\sigma_M^2(u)|_{A(u)}>\delta,\\
\{\nabla J_\text{vol}+\alpha\nabla J_\sigma^x|\;x\in A(u)\}\cup\{\nabla J_\text{vol}\},&\quad\text{ if }\sigma_M^2(u)|_{A(u)}=\delta,\\
\{\nabla J_\text{vol}\},&\quad\text{ if }\sigma_M^2(u)|_{A(u)}<\delta,
\end{cases}
\end{equation}
this further simplifies to
\begin{equation}
DJ(\Omega)(X)=\max_{L\in \mathcal Z}\langle L,X\rangle_{\mathcal H}.
\end{equation}
The next lemma shows that we can extend this result to the closed convex hull of $\mathcal Z$.
\begin{lemma}\label{lem:conv_hull}
Let $\mathcal L:=\overline{\text{conv}(\mathcal Z)}$. Then there holds
\begin{equation}
DJ(\Omega)(X)=\max_{L\in \mathcal L}\langle L,X\rangle_{\mathcal H},
\end{equation}
for all $X\in \mathcal H$.
\end{lemma}
\begin{proof}
For a proof we refer to \cite[Lemma 3.4]{a_ST_2016a}
\end{proof}
The next Lemma characterises the steepest descent direction.
\begin{lemma}\label{lem:steepest_descent}
Assume $0\notin \mathcal L$. Then there holds
\begin{equation}
\min_{\|X\|_{\mathcal H}=1} DJ(\Omega)(X)=-\|X^\ast\|_{\mathcal H},
\end{equation}
where $X^\ast=-\frac{Z^\ast}{\|Z^\ast\|_{\mathcal H}}$ and $Z^\ast=\argmin_{Z\in\mathcal L} \|Z\|_{\mathcal H}$.
\end{lemma}
\begin{proof}
A proof can be found in \cite[Lemma 3.3]{b_DEMALO_2014a}
\end{proof}
As a consequence of the previous theorem we can deduce the following optimality condition.
\begin{corollary}
The set $\Omega\in \mathcal A$ is a local minimiser of the penalised objective functional $J$ if and only if $0\in \mathcal L$. In this context $\Omega$ is said to be a local minimum of $J$ iff there is $\rho>0$ such that $J(\Omega)\le J\left((\text{Id}+X)(\Omega)\right)$ for all $X\in \mathcal H$, $\|X\|_{\mathcal H}<\rho$.
\end{corollary}
\begin{proof}
From Lemma \ref{lem:steepest_descent} we know that $DJ(\Omega)(X^\ast)<0$ if $0\notin\mathcal L$. In contrast, if $0\in\mathcal L$ there holds $DJ(\Omega)(X)\ge \langle 0,X\rangle_{\mathcal H}=0$ for all $X\in\mathcal H$.
\end{proof}
\begin{remark}
Expanding the case dependent definition of $\mathcal Z$ in \eqref{eq:L_cases} leads to the following characterisations of the optimality conditions. If the stress constraint is strictly not violated, i.e. if $\sigma_M^2(u)|_K<\delta$, $\mathcal L= \{\nabla J_\text{vol}\}$ and thus the optimality condition reads $\nabla J_\text{vol}=0$. Furthermore, if the stress threshold is surpassed, i.e. if there is $x\in K$ such that $\sigma_M^2(u)(x)>\delta$, the optimality condition yields $0\in \overline{\text{conv}(\mathcal Z)}$. Finally, in the case where the stress constraint is active, i.e. if there is $x\in A(u)$ such that $\sigma_M^2(u)(x)=\delta$, the optimality condition includes the previous cases $\nabla J_\text{vol}=0$ and $0\in \overline{\text{conv}(\mathcal Z)}$.
\end{remark}
\section{Clarke subgradient}\label{sec:clarke}
In this section we are going to connect the set $\mathcal L$ defined in Lemma \ref{lem:conv_hull} to the Clarke subgradient \cite{b_Clarke_1983}.
\begin{definition}
Let $(\mathcal X,\|\cdot\|)$ be a Banach space and $f:\mathcal X\to \VR$ a function. The generalised directional derivative of $f$ at $x\in \mathcal X$ in direction $v\in\mathcal X$ is defined as
\begin{equation}
f^\circ(x;v):=\limsup_{\substack{y\to x\\ t\searrow 0}}\frac{f(y+tv)-f(y)}{t}.
\end{equation}
Furthermore, the Clarke subgradient at $x_0\in \mathcal X$ is defined as the set
\begin{equation}
\partial^\circ f(x_0):=\{x^\ast\in \mathcal X^\ast|\;f^\circ(x_0;v)\ge x^\ast(v),\text{ for all }v\in \mathcal X\}.
\end{equation}
\end{definition}
In order to fit into the framework, we redefine our objective functional as follows.
\begin{definition}
Let $\Omega\in\mathcal A$ fix and $\mathcal H$ denote the Hilbert space defined in the previous section. Furthermore, let $\rho>0$ sufficiently small. The function $G:B_\rho(0)\to\VR$ is defined by
\begin{equation}
G(X):=J((\text{Id}+X)(\Omega)).
\end{equation}
Additionally, we define for $x\in K$ the pointwise function $G^x:B_\rho(0)\to\VR$ by
\begin{equation}
G^x(X):=J_\text{vol}((\text{Id}+X)(\Omega))+\alpha J^x_\sigma((\text{Id}+X)(\Omega)),
\end{equation}
where $\alpha>0$ denotes the penalty parameter.
\end{definition}
\begin{remark}
Note that the functions are indeed well defined for $\|X\|_{\mathcal H}$ sufficiently small. Furthermore, there holds
\begin{equation}
D^+G(0,X)=DJ(\Omega)(X)=\max_{L\in\mathcal L}\langle L,X\rangle_{\mathcal H},
\end{equation}
where $D^+G(0,X)$ denotes the directional derivative of $G$ at $0\in\mathcal H$ in direction $X\in\mathcal H$. A similar result holds true for $D^+G^x(0,X)$.
\end{remark}
Since our definition of the generalised derivative does not require $G$ to be locally Lipschitz, our next lemma shows that the quantity is nonetheless finite.
\begin{lemma}\label{lem:clarke_bound}
Let $Y\in\mathcal H$ sufficiently close to $0\in\mathcal H$ and $t>0$ small. Then there exists $C\in\VR$ such that
\begin{equation}
\left|\frac{G(Y+tX)-G(Y)}{t}\right|\le C.
\end{equation}
\end{lemma}
\begin{proof}
Let $x^{Y,t}\in A(u^{Y,t})$. Then there holds by the maximising property of the set $A(u^{Y,t})$.
\begin{align}\label{eq:clarke_quot_bound}
\frac{G(Y+tX)-G(Y)}{t}\le&\frac{G^{x^{Y,t}}(Y+tX)-G^{x^{Y,t}}(Y)}{t}\\
=&\frac{B\eps_{Y,t}(u^{Y,t})(x^{Y,t}):\eps_{Y,t}(u^{Y,t})(x^{Y,t})-B\eps_{Y}(u^{Y})(x^{Y,t}):\eps_{Y}(u^{Y})(x^{Y,t})}{t}\\
=&B[\frac{\eps_{Y,t}-\eps_Y}{t}](u^{Y,t})(x^{Y,t}):\eps_{Y,t}(u^{Y,t})(x^{Y,t})\\
&+B\eps_Y(\frac{u^{Y,t}-u^Y}{t})(x^{Y,t}):\eps_{Y,t}(u^{Y,t})(x^{Y,t})\\
&+B\eps_Y(u^Y)(x^{Y,t}):[\frac{\eps_{Y,t}-\eps_Y}{t}](u^{Y,t})(x^{Y,t})\\
&+B\eps_Y(u^Y)(x^{Y,t}):\eps_Y(\frac{u^{Y,t}-u^Y}{t})(x^{Y,t}).
\end{align}
Assumption \ref{ass:material_c1} entails $\frac{u^{Y,t}-u^Y}{t}\to \dot u$, $u^{Y,t}\to u$ and $u^Y\to u$ in $C^1(K)$. Furthermore, Lemma \ref{lem:derivatives} item (ii) yields the uniform convergences of $\frac{\alpha_{Y,t}-\alpha_Y}{t}$, $\alpha_{Y,t}$ and $\alpha_Y$. Thus, we conclude that for sufficiently small $Y\in\mathcal H$ and $t>0$ the right-hand side of equation \eqref{eq:clarke_quot_bound} remains bounded. Using $x^Y\in A(u^Y)$ in a similar way shows that
\begin{equation}
\frac{G(Y)-G(Y+tX)}{t}\le C.
\end{equation}
This concludes the proof.
\end{proof}
Next we show that for our objective functional the directional derivative and the generalised directional derivative coincide.
\begin{lemma}\label{lem:equality_derivatives}
Let $X\in\mathcal H$. Then there holds
\begin{equation}
D^+G(0,X)=D^\circ G(0,X).
\end{equation}
\end{lemma}
\begin{proof}
We show this result in two steps. For the inequality ``$\le$'' note that by the definition of the generalised directional derivative there holds
\begin{equation}
D^+G(0,X)=\lim_{t\searrow 0}\frac{G(tX)-G(0)}{t}\le \limsup_{\substack{Y\to 0\\ t\searrow 0}}\frac{G(Y+tX)-G(Y)}{t}=D^\circ G(0,X).
\end{equation}
For the converse note that by Lemma \ref{lem:clarke_bound} the quotient appearing in the generalised direction derivative remains bounded. Hence, there are $Y_k\in\mathcal H$, $t_k>0$ such that $Y_k\to 0$, $t_k\searrow 0$ for $k\to\infty$ and
\begin{equation}
D^\circ G(0,X)=\limsup_{\substack{Y\to 0\\ t\searrow 0}}\frac{G(Y+tX)-G(Y)}{t}=\lim_{k\to\infty}\frac{G(Y_k+t_kX)-G(Y_k)}{t_k}.
\end{equation}
Next we pick for each $k\in\VN$ a point $x^k\in A(u^{Y_k,t_k})\subset K$. Since $K$ is compact, there exists a subsequence, which we denote the same, and $x\in K$ such that $x_k\to x$. This yields
\begin{equation}
D^\circ G(0,X)=\lim_{k\to\infty}\frac{G(Y_k+t_kX)-G(Y_k)}{t_k}\le\lim_{k\to\infty}\frac{G^{x_k}(Y_k+t_kX)-G^{x_k}(Y_k)}{t_k}.
\end{equation}
Now expanding the right-hand side as in \eqref{eq:clarke_quot_bound} and using the same arguments to pass to the limits, shows
\begin{equation}
D^\circ G(0,X)\le\lim_{k\to\infty}\frac{G^{x_k}(Y_k+t_kX)-G^{x_k}(Y_k)}{t_k}=D^+G^x(0,X).
\end{equation}
Next we note that for each $y\in K$ there holds $G^y(Y_k+t_kX)\le G^{x_k}(Y_k+t_kX)$. Hence, passing to the limit $k\to\infty$ yields
\begin{equation}
G^y(0)\le G^x(0),
\end{equation}
i.e. $x\in A(u)$. Hence, putting our observations together, yields
\begin{equation}
D^\circ G(0,X)\le D^+G^x(0,X)\le \max_{z\in A(u)} D^+G^z(0,X)=D^+G(0,X),
\end{equation}
which concludes the proof.
\end{proof}
We are now able to prove the main result of this section
\begin{theorem}
Let $\mathcal L$ be defined in Lemma \ref{lem:conv_hull} and $\partial^\circ G(0)$ denote the Clarke subgradient of $G$ at $0\in\mathcal H$. Then there holds
\begin{equation}
\mathcal L=\partial^\circ G(0).
\end{equation}
\end{theorem}
\begin{proof}
First note that by the definition of the generalised directional derivative and Lemma \ref{lem:equality_derivatives} there holds for every $X\in \mathcal H$
\begin{equation}
\langle L,X\rangle_{\mathcal H}\le \max_{L\in\mathcal L}\langle L,X\rangle_{\mathcal H}=D^+G(0,X)=D^\circ G(0,X),\quad\text{ for all }L\in\mathcal L.
\end{equation}
Thus, we conclude $\mathcal L\subset \partial^\circ G(0)$. For the converse let $Z\in\mathcal{L}^c$. By the Hahn-Banach separation theorem there exists a linear continuous functional $\mu$ and a constant $c\in\VR$ such that $\mu(Z)>c>\mu(L)$ for all $L\in\mathcal L$. By the Riesz-representation we can identify $\mu$ with an element $\bar{X}\in\mathcal H$. Hence, it follows
\begin{equation}
\langle Z,\bar{X}\rangle_{\mathcal H}>c>\langle L,\bar{X}\rangle_{\mathcal H},\quad\text{ for all }L\in\mathcal L.
\end{equation}
Taking the maximum over $L\in\mathcal L$ now shows
\begin{equation}
\langle Z,\bar{X}\rangle_{\mathcal H}>D^\circ G(0,\bar{X}),
\end{equation}
where we used again that the directional derivative and the generalised directional derivative coincide. Hence, we deduce $Z\notin \partial G^\circ(0)$, which concludes the proof.
\end{proof}
\section{Numerical implementation}\label{sec:numerics}
In this section we give some numerical results. In contrast to the initial problem formulation in Section \ref{sec:introduction}, we add an additional weighting factor to the volume term. This allows a more flexible numerical treatment. Furthermore, one can use this setting to interpret the problem as stress minimisation with a volume constraint.
\paragraph{Shape derivative and shape gradient}
In order to identify shape derivatives with gradients, we use the following approach. Given a shape derivative $D\tilde J(\Omega)(\cdot)$ we define $\nabla \tilde J\in H^1_{\Gamma\cup\Gamma^N}(\Omega)^d$ as the unique solution of the sub-problem
\begin{equation}\label{eq:hilbert_projection}
\int_\Omega \eps(\nabla\tilde J):\eps(X)+B\nabla \tilde J\cdot BX+\rho_\text{low}\nabla \tilde J\cdot X\;dx=D\tilde J(\Omega)(X),\quad\text{ for all }X\in H^1_{\Gamma\cup\Gamma^N}(\Omega)^d.
\end{equation}
Here, $\rho_\text{low}>0$ is a constant and the term
\begin{equation}
B=
\begin{pmatrix}
-\partial_x&\partial_y\\\partial_y&\partial_x
\end{pmatrix}
\end{equation}
incorporates the Cauchy-Riemann equations and thus encourages the solution to be a conformal mapping, which accounts for a good mesh quality. We would like to point out that by definition $\nabla\tilde J\in H^1_{\Gamma\cup\Gamma^N}(\Omega)^d$ and is thus not necessary differentiable. This is definitely a conflict in the analytic setting, where enough regularity is required for the exact derivation of the derivatives. Nevertheless, this should not cause any problems in the numerical realm.
\paragraph{Moving mesh method and mesh quality}
In the upcoming numerical implementation we follow the ``moving mesh'' approach, i.e. in each iteration we deform the initial mesh with respect to a descent direction given by a vector field. Additionally, we check the mesh quality in each iteration and perform a remeshing if it is required. In this context, mesh quality is related to the quotient $\frac{r_i^E}{r_o^E}$ of the inner radius $r_i^E$ and outer radius $r_o^E$ of each triangular element $E$.
\paragraph{Differences to the analytic setting}
In contrast to the problem formulation in Section \ref{sec:introduction}, we discard the subset $\omega$ for the numerical implementation. This can be justified, since we are working in a finite dimensional finite element space. Additionally, we only consider deformation vector fields that vanish on the boundary $\Gamma^N$, where the force is applied, as well. We would like to point out that, contrary to the initial problem formulation, we consider an additional free boundary. Therefore, the assumption $X\equiv0$ on $\Gamma\cup\Gamma^N$ does not fix the whole body. Finally we want to mention that, contrary to the initial problem formulation, the domains occurring in the upcoming numerical examples include corners on the boundary and thus lack some regularity that is required in the analytical setting.
\subsection{max-norm approach}
Our goal is to minimise the functional
\begin{equation}\label{eq:cost_numerics_nonsmooth}
J(\Omega):=\gamma_1 \left(|\Omega|-V\right)^2+\gamma_2\max\{\max_{x\in\bar{\Omega}}\sigma_M^2(u_\Omega)-\delta,0\},
\end{equation}
where $\Omega\in \mathcal A$, $u_\Omega\in H^1_\Gamma(\Omega)^2$ solves \eqref{eq:state}, $\gamma_1,\gamma_2\in\VR^+$ are given weights and $\delta=0$ denotes the stress threshold. Here, $|\Omega|$ denotes the volume of $\Omega$ and the constant $V\in \VR$ is a given target volume.  We observe that for each $x\in\bar{\Omega}$ and $r>0$, plugging in the approximated adjoint variable $q^{x,r}$ defined in \eqref{eq:adjoint_approx} yields
\begin{align}
\begin{split}
D J_\sigma^x(\Omega)(X)\approx&-B[\partial u\partial X+(\partial X)^\top (\partial u)^\top]:\eps(u)(x)\\
&\frac{1}{2}\int_\Omega A [\partial u\partial X+\partial X^\top\partial u^\top]:\eps(q^{x,r})\; dx-\int_\Omega \Div XA\eps(u):\eps(q^{x,r})\; dx\\
&+\frac{1}{2}\int_\Omega A \eps(u):[\partial q^{x,r}\partial X+\partial X^\top(\partial q^{x,r})^\top]\; dx\\
&+\int_\Omega \Div X f\cdot q^{x,r}\; dx+\int_\Omega \partial f X\cdot q^{x,r}\; dx\\
&+\int_{\Gamma^N} (\Div X-(\partial Xn)\cdot n) g\cdot q^{x,r}\; dS+\int_{\Gamma^N} \partial g X\cdot q^{x,r}\; dS.
\end{split}
\end{align}
With this notation we introduce the following optimisation algorithm that builds on the results of Section \ref{sec:hilbert_setting}.
\begin{algorithm}[H]
\caption{Basic gradient algorithm - max-norm approach}
\label{alg:nonsmooth}
\begin{algorithmic}
    \REQUIRE initial shape $\Omega_0\subset \VR^2$, $N_{max}\in \VN$
    \WHILE{$n\le N_{max}$}
    \STATE{choose Hilbert space $\Ch_n= H^1_{\Gamma\cup\Gamma^N}(\Omega_n)^2$}
    
    \STATE{check mesh quality and remesh if required}
    \STATE{compute state $u_n$ by solving state equation \eqref{eq:state} on $\Omega_n$ }
    \STATE{compute active set $A(u_n)=\{x_1,...,x_k\}$}
    \STATE{compute adjoint states $q^{x_i,r}_n$ \eqref{eq:adjoint_approx}, $i\in\{1,...,k\}$ on domain $\Omega_n$ with state $u_n$}
    \STATE{compute gradients $\nabla J_\text{vol}, \nabla J_2^{x_i}\in Ch_n$, for $x_i\in A(u_n)$ }
\STATE{assemble $A\in\VR^{k\times k}$ with $a_{ij}:= (\gamma_1 \nabla J_\text{vol}+\gamma_2\nabla J_\sigma^{x_i},\gamma_1 \nabla J_\text{vol}+\gamma_2\nabla J_\sigma^{x_j})_{\Ch_n}$ for $i,j\in \{1,\ldots, k\}$}
\STATE{solve $\min_{\alpha\in \VR^k} A\alpha\cdot\alpha$ subject to $\sum_{i=1}^k\alpha_i=1$}
\STATE{  $Z_n:= \sum_{i=1}^k \alpha_i v_i$.}
\STATE descent direction $X_n := - Z_n$
\STATE{choose step size $s_n>0$}
\IF{$J((\Id+s_nX_n)(\Omega_n))>J(\Omega_n)$ and no remeshing occurred in the previous step}
\STATE{remesh}
\ELSE
\STATE{deform shape $\Omega_{n+1}:= (\Id+s_nX_n)(\Omega_n)$}
\ENDIF
\ENDWHILE
\end{algorithmic}
\end{algorithm}

\subsection{$p$-norm approach}
In order to validate our method we compare our results to the following $p$-norm approach, which includes a smooth regularisation of the stress term:
\begin{equation}\label{eq:cost_numerics_smooth}
J(\Omega):=\gamma_1 \left(|\Omega|-V\right)^2+\gamma_2\left(\int_\Omega |\sigma_M(u_\Omega)|^p\;dx\right)^\frac{1}{p},
\end{equation}
for $p\ge2$.  The shape derivative of the $J_p(\Omega):=\left(\int_\Omega |\sigma_M(u_\Omega)|^p\;dx\right)^\frac{1}{p}$ for $2\le p<\infty$  proved by various techniques as reported in \cite{c_ST_2015a}. We are omitting the proof here and only state the derivative. 
For $X \in C^1(\bar\Dsf)^d$ with $\text{supp} (X)\subset (\Dsf\setminus\bar{\omega})$ we have
\begin{align}
\begin{split}
DJ_p(\Omega)(X)=&c_\Omega \int_\Omega \Div X|\sigma_M(u_\Omega)|^p\;dx\\
&-c_\Omega \frac{p}{2}\left(B\eps(u):\eps(u)\right)^{\frac{p}{2}-1}B[\partial u\partial X+(\partial X)^\top (\partial u)^\top]:\eps(u)\;dx\\
&+\frac{1}{2}\int_\Omega A [\partial u\partial X+\partial X^\top\partial u^\top]:\eps(q)\; dx-\int_\Omega \Div XA\eps(u):\eps(q)\; dx\\
&+\frac{1}{2}\int_\Omega A \eps(u):[\partial q\partial X+\partial X^\top(\partial q)^\top]\; dx\\
&+\int_\Omega \Div X f\cdot q\; dx+\int_\Omega \partial f X\cdot q\; dx\\
&+\int_{\Gamma^N} (\Div X-(\partial Xn)\cdot n) g\cdot q\; dS+\int_{\Gamma^N} \partial g X\cdot q\; dS.
\end{split}
\end{align}
where $c_\Omega=\frac{1}{p}\left(\int_\Omega |\sigma_M(u_\Omega)|^p\;dx\right)^\frac{1-p}{p}$ and $q\in H^1_\Gamma(\Omega)^d$ solves
\begin{equation}\label{eq:adjoint_pnorm}
\int_\Omega A\eps(\varphi):\eps(q)\;dx=c_\Omega p\left(B\eps(u):\eps(u)\right)^{\frac{p}{2}-1}B\eps(u):\eps(\varphi)\;dx,\quad\text{ for all }\varphi\in H^1_\Gamma(\Omega)^d.
\end{equation}
\begin{algorithm}[H]
\caption{Basic gradient algorithm - $p$-norm approach}
\label{alg:smooth}
\begin{algorithmic}
    \REQUIRE initial shape $\Omega_0\subset \VR^2$, $N_{max}\in \VN$
    \WHILE{$n\le N_{max}$}
    \STATE{choose Hilbert space $\Ch_n= H^1_{\Gamma\cup\Gamma^N}(\Omega_n)^2$}
    \STATE{check mesh quality and remesh if required}
    \STATE{compute state $u_n$ by solving state equation \eqref{eq:state} on $\Omega_n$ }
    \STATE{compute adjoint state $q_n$ by solving adjoint state equation \eqref{eq:adjoint_pnorm} on $\Omega_n$ with state $u_n$}
    \STATE{compute gradients $\nabla J_\text{vol}, \nabla J_p\in\Ch_n$ }
\STATE{descent direction $X_n := -\left(\gamma_1\nabla J_\text{vol}+\gamma_2\nabla J_p\right)$}

\STATE{choose step size $s_n>0$}
\IF{$J((\Id+s_nX_n)(\Omega_n))>J(\Omega_n)$ and no remeshing occurred in the previous step}
\STATE{remesh}
\ELSE
\STATE{deform shape $\Omega_{n+1}:= (\Id+s_nX_n)(\Omega_n)$}
\ENDIF
\ENDWHILE
\end{algorithmic}
\end{algorithm}
\subsection{Numerical results}
We implemented both algorithms in the software NGSolve \cite{a_SC_2014}. In order to achieve comparable results, we normalised the respective stress gradients in each iteration and moved a fixed step size. Throughout our numerical experiments we chose the material parameter as follows: the Poisson ration $\nu=0.3$ and the Young modulus $E=1$. These in turn yield the Lam\'{e} coefficients
\begin{equation}
\lambda=\frac{\nu E}{(1+\nu)(1-2\nu)},\qquad \mu=\frac{E}{2(1+\nu)}.
\end{equation}
Additionally we consider the absence of a volume force, i.e. $f\equiv 0$, which allows a simplified representation of the involved terms.
\subsubsection{L-bracket}
For the first example we considered the L-bracket problem. Here, the initial set is given as
\begin{equation}
\Omega=(0,100)\times (0,100)\setminus [40,100]\times[40,100].
\end{equation}
The bracket is fixed on the upper part of the boundary $\Gamma=[0,40]\times\{100\}$ and the boundary force $g=(0,-3)^\top$ is applied at the corner of the rightmost boundary part, i.e.
\begin{equation}
\Gamma^N=\{100\}\times[35,40]\cup[95,100]\times\{40\}.
\end{equation} 
Furthermore, the target volume was fixed as $70\%$ of the initial volume. The setting is depicted in Figure \ref{fig:lshape_sketch}. The remaining unknowns were chosen as follows: $p=6$, $\gamma_1=10^{-4}$, $\gamma_2=1$, $\rho_\text{low}=10^{-2}$, $r=100$, the initial meshsize $h=5$ and finite elements of order $3$.
\begin{figure}[H]
\begin{center}
\begin{tikzpicture}[every node/.append style={fill=white}]
\fill [
       pattern={Lines[distance=1mm,
                  angle=45,
                  line width=0.25mm
                 ]},
        pattern color=black
       ] (0,5) rectangle (2,5.35);
       \draw [latex-] (0,-0.5) -- (2.5,-0.5) node [below]{\small 100};
       \draw [-latex] (2.5,-0.5) -- (5,-0.5);
       \draw [latex-] (2.31,2.3) -- (3.3,2.3) node [above]{\small 55};
       \draw [-latex] (3.3,2.3) -- (4.2,2.3);
       \draw [latex-] (4.2,2.3) -- (4.6,2.3) node [above]{\small 5};
       \draw [-latex] (4.6,2.3) -- (5,2.3);
       \draw [latex-] (2.3,2.31) -- (2.3,3.5) node [right]{\small 60};
       \draw [-latex] (2.3,3.5) -- (2.3,5);
       \draw [latex-] (-0.5,0) -- (-0.5,2.5) node [left]{\small 100};
       \draw [-latex] (-0.5,2.5) -- (-0.5,5);
       \draw [latex-] (5.5,0) -- (5.5,0.6) node [right]{\small 35};
       \draw [-latex] (5.5,0.6) -- (5.5,1.2);
       \draw [line width=0.35mm, black ] (0,0) -- (5,0) -- (5,2) -- (2,2) -- (2,5) -- (0,5) -- (0,0);
       \draw node at (1,4.7) {$\Gamma$};
       \draw node at (4.6,1.6) {$\Gamma^N$};
       \draw node at (5,4) {$g$};
       \draw node at (1.15,1.15) {\Large$\Omega$};
       \draw [line width=1mm,-latex] (4.65,3.6) -- (4.65,2.9);
       \draw [line width=1mm,-latex] (5,3.6) -- (5,2.9);
       \draw [line width=1mm,-latex] (5.35,3.6) -- (5.35,2.9);
\end{tikzpicture}
\end{center}
\caption{Visualisation of the L-bracket}
\label{fig:lshape_sketch}
\end{figure}
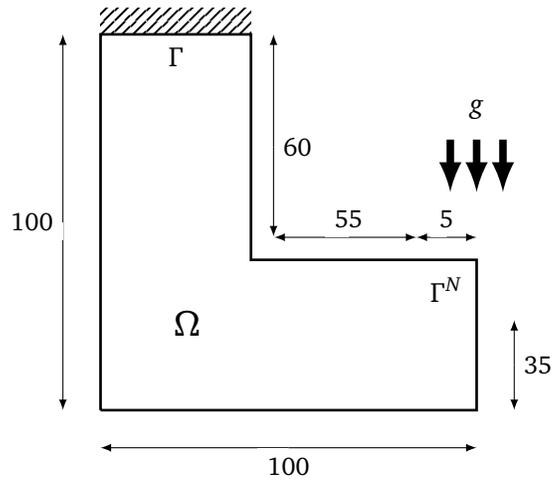

The resulting deformation of the shape are visualised in Figure \ref{fig:L_deformation} and the corresponding evolution of the cost functionals is given in Figure \ref{fig:L_cost}

\begin{figure}[H]

\begin{subfigure}{.249\linewidth}
  \includegraphics[width=\linewidth]{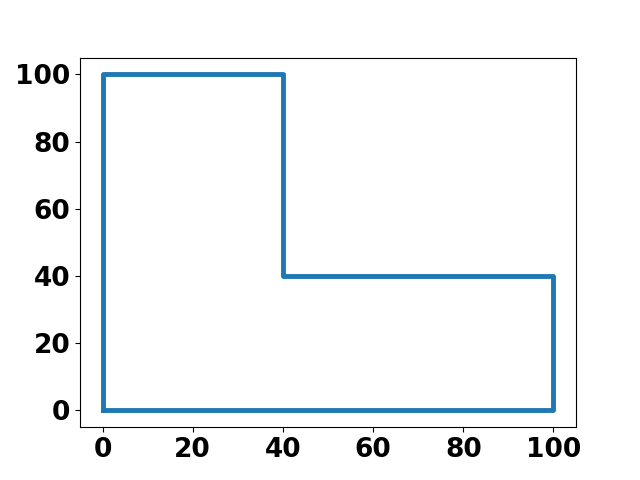}
  \caption{max-norm - It: 0}
  \label{fig:nonsmooth_0it}
\end{subfigure}\hfill 
\begin{subfigure}{.249\linewidth}
  \includegraphics[width=\linewidth]{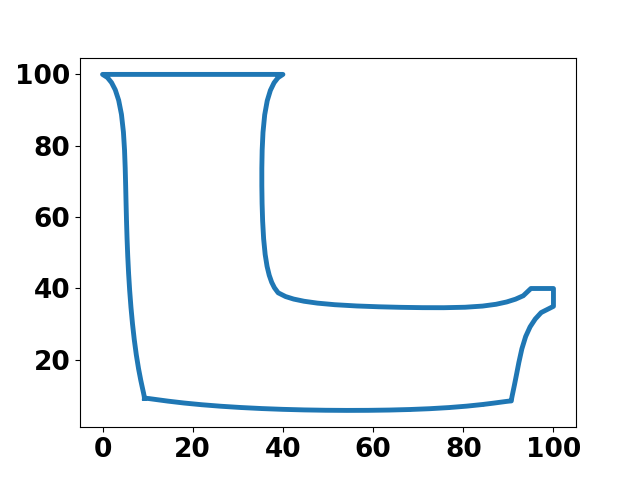}
  \caption{max-norm - It: 10}
  \label{fig:nonsmooth_10it}
\end{subfigure}\hfill 
\begin{subfigure}{.249\linewidth}
  \includegraphics[width=\linewidth]{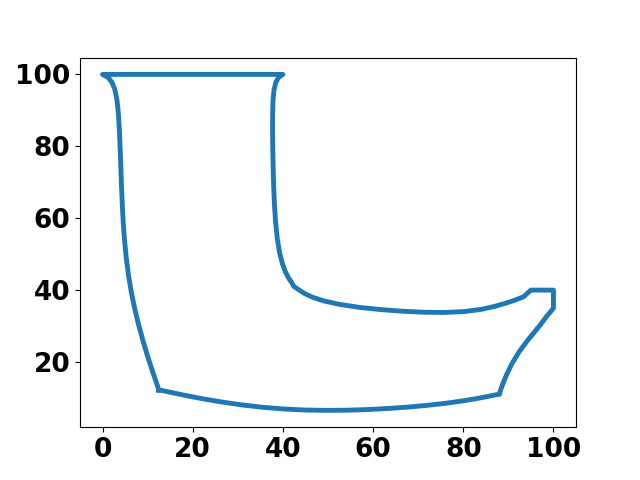}
  \caption{max-norm - It: 100}
  \label{fig:nonsmooth_100it}
\end{subfigure}\hfill 
\begin{subfigure}{.249\linewidth}
  \includegraphics[width=\linewidth]{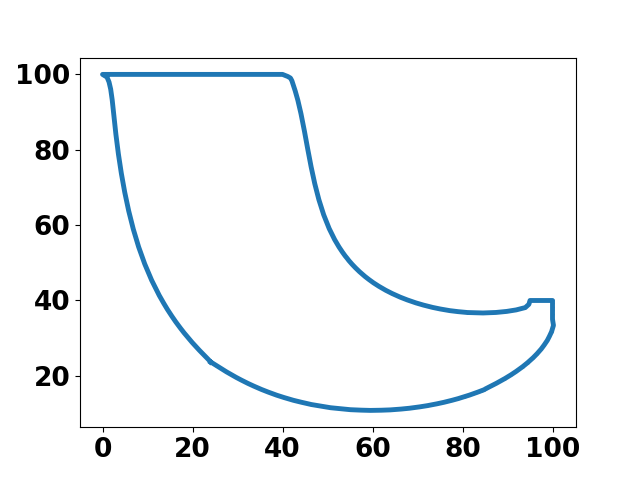}
  \caption{max-norm - It: 560}
  \label{fig:nonsmooth_560it}
\end{subfigure}

\medskip 

\begin{subfigure}{.249\linewidth}
  \includegraphics[width=\linewidth]{nonsmooth_0it.png}
  \caption{$p$-norm - It: 0}
  \label{fig:smooth_0it}
\end{subfigure}\hfill 
\begin{subfigure}{.249\linewidth}
  \includegraphics[width=\linewidth]{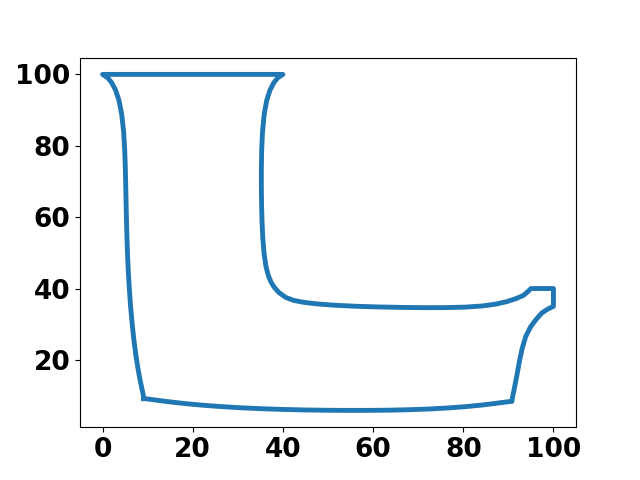}
  \caption{$p$-norm - It: 10}
  \label{fig:smooth_10it}
\end{subfigure}\hfill 
\begin{subfigure}{.249\linewidth}
  \includegraphics[width=\linewidth]{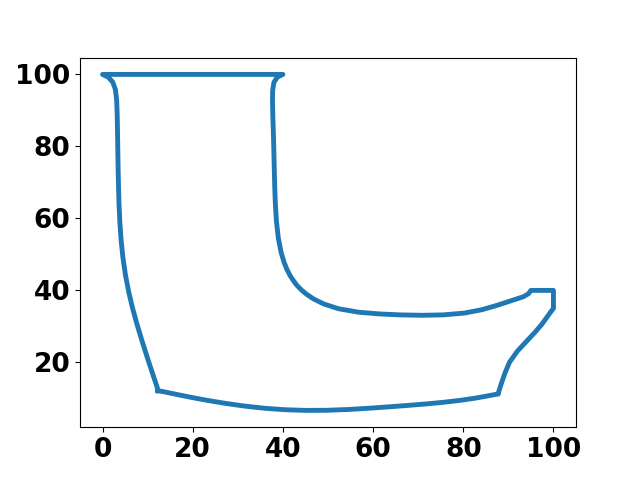}
  \caption{$p$-norm - It: 100}
  \label{fig:smooth_100it}
\end{subfigure}\hfill 
\begin{subfigure}{.249\linewidth}
  \includegraphics[width=\linewidth]{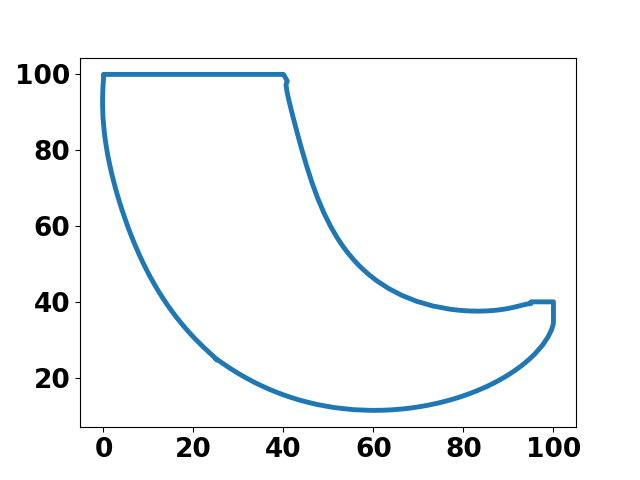}
  \caption{$p$-norm - It: 560}
  \label{fig:smooth_560it}
\end{subfigure}

\caption{Deformation of the L-bracket}
\label{fig:L_deformation}
\end{figure}

\begin{figure}[H]

\begin{subfigure}{.32\linewidth}
  \includegraphics[width=\linewidth]{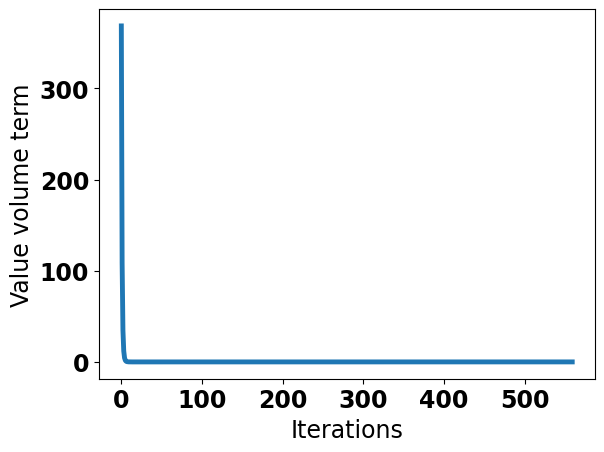}
  \caption{max-norm: volume cost}
  \label{fig:nonsmooth_hook_volcost}
\end{subfigure}\hfill 
\begin{subfigure}{.32\linewidth}
  \includegraphics[width=\linewidth]{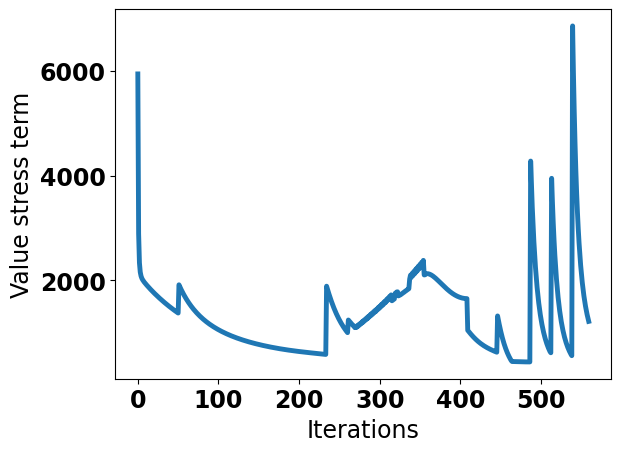}
  \caption{max-norm: stress cost}
  \label{fig:nonsmooth_hook_stresscost}
\end{subfigure}\hfill 
\begin{subfigure}{.32\linewidth}
  \includegraphics[width=\linewidth]{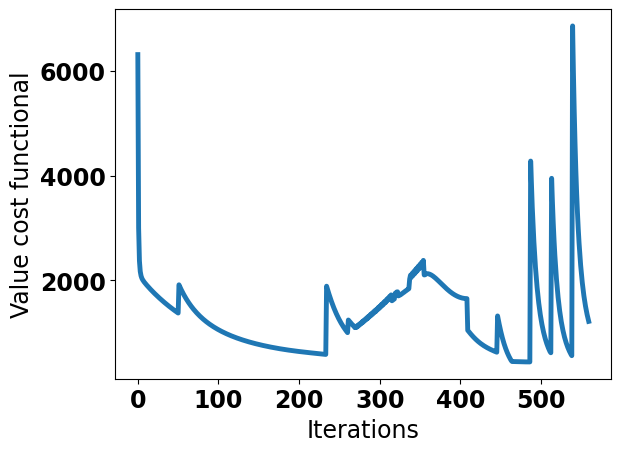}
  \caption{max-norm: total cost}
  \label{fig:nonsmooth_hook_fullcost}
\end{subfigure}

\medskip 

\begin{subfigure}{.32\linewidth}
  \includegraphics[width=\linewidth]{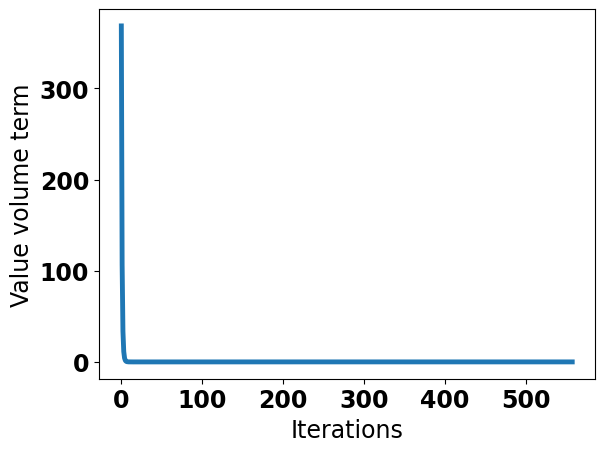}
  \caption{$p$-norm: volume cost}
  \label{fig:smooth_hook_volcost}
\end{subfigure}\hfill 
\begin{subfigure}{.32\linewidth}
  \includegraphics[width=\linewidth]{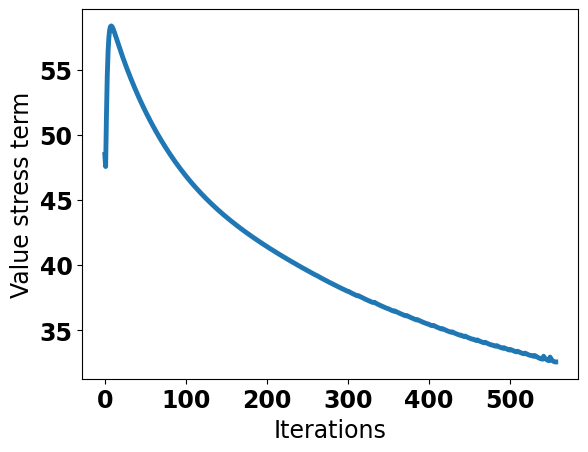}
  \caption{$p$-norm: stress cost}
  \label{fig:smooth_hook_stresscost}
\end{subfigure}\hfill 
\begin{subfigure}{.32\linewidth}
  \includegraphics[width=\linewidth]{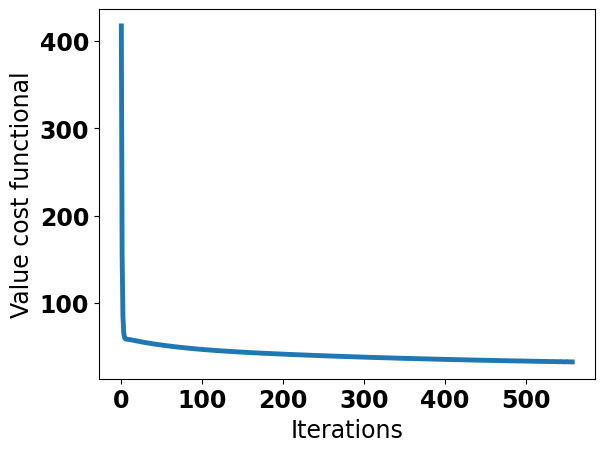}
  \caption{$p$-norm: total cost}
  \label{fig:smooth_hook_fullcost}
\end{subfigure}

\caption{Cost evolution for the L-bracket problem}
\label{fig:L_cost}
\end{figure}
To compare both methods in terms of the pointwise stress optimisation, we deduced the minimal meshsizes $h^\text{min}_{p\text{-norm}}$, $h^\text{min}_\text{max-norm}$ of both algorithms after a given number of iterations. We then chose $h^\text{min}=\min\{h^\text{min}_{p\text{-norm}}, h^\text{min}_\text{max-norm}\}$ and remeshed both final shapes with the minimal meshsize $h^\text{min}$. Finally, we solved the PDE on both domains and compared the resulting maximum von Mises stress. This uniform meshsize is necessary, since the maximum stress highly depends on the underlying meshsize. The results for some iterations are listed in Table \ref{tab:L_comp}.

\begin{table}[H]
\begin{center}
\begin{tabular}{ |c||c|c| } 
 \hline
  & max-norm & $p$-norm \\ 
 \hline\hline
 10 iterations & 1258 & 1466 \\ 
 \hline
 30 iterations & 691 & 771 \\ 
 \hline
  50 iterations & 464 & 615 \\ 
 \hline
  100 iterations & 384 & 391 \\ 
 \hline
   140 iterations & 262 & 330 \\ 
 \hline
\end{tabular}
\end{center}
\caption{Comparison of the maximal von Mises stress $\sigma_M^2$}
\label{tab:L_comp}
\end{table}
Comparing the final results (Figure \ref{fig:nonsmooth_560it}, Figure \ref{fig:smooth_560it}) we observe that visually both algorithms yield similar results. These in fact coincide with the results obtained in  \cite{a_PITO_2018}, where the $p$-norm problem formulation was investigated with a different numerical approach. While both approaches approximate the target volume within a few iterations (Figure \ref{fig:nonsmooth_hook_volcost}, Figure \ref{fig:smooth_hook_volcost}), the stress-cost evolves differently. In the $p$-norm approach, the value of the stress functional shows a steady decrease until the deformation causes some minor artificial fluctuations towards the end (see Figure \ref{fig:smooth_hook_stresscost}). In contrast, the maximum stress curve in the max-norm approach (see Figure \ref{fig:nonsmooth_hook_stresscost}) shows a number of peaks. These are linked to the occurrence of remeshes, as the refined mesh yields higher stress values. Finally, Table \ref{tab:L_comp} shows that both algorithms yield a decrease of the maximal von Mises stress, yet the max-norm approach seems to perform slightly better in this regard.
\subsubsection{Bridge}
For the second example we consider a bridge with vertical load placed in the center of the upper boundary. To be precise, the initial set is given as
\begin{equation}
\Omega=(0,5)\times (0,5)\setminus [1,4]\times[0,2].
\end{equation}
The bridge is fixed at the bottom boundary $\Gamma=\left([0,1]\times\{0\}\right)\cup\left([4,5]\times\{0\}\right)$ and the force $g=(0,-3)^\top$ is applied on $\Gamma^N=[2,3]\times\{5\}$. Again, the target volume was fixed as $70\%$ of the initial volume. A schematic of the setting can be seen in Figure \ref{fig:bridge_sketch}. For this example we chose the regularisation $p=2$. Furthermore, the remaining parameter were set $\gamma_1=10^{-2}$, $\gamma_2=1$, $\rho_\text{low}=10$, $r=7.5$, the initial meshsize $h=0.5$ and finite elements of order $3$.
\begin{figure}[H]
\begin{center}
\begin{tikzpicture}[every node/.append style={fill=white}]
\fill [
       pattern={Lines[distance=1mm,
                  angle=45,
                  line width=0.25mm
                 ]},
        pattern color=black
       ] (0,-0.35) rectangle (1,0);
       \fill [
       pattern={Lines[distance=1mm,
                  angle=45,
                  line width=0.25mm
                 ]},
        pattern color=black
       ] (4,-0.35) rectangle (5,0);
       \draw [latex-] (1,2.35) -- (2.5,2.35) node [above]{\small 3};
       \draw [-latex] (2.5,2.35) -- (4,2.35);
       \draw [latex-] (0,-0.7) -- (0.5,-0.7) node [below]{\small 1};
       \draw [-latex] (0.5,-0.7) -- (1,-0.7);
       \draw [latex-] (4,-0.7) -- (4.5,-0.7) node [below]{\small 1};
       \draw [-latex] (4.5,-0.7) -- (5,-0.7);
       \draw [latex-] (0,4.65) -- (1,4.65) node [below]{\small 2};
       \draw [-latex] (1,4.65) -- (2,4.65);
       \draw [latex-] (3,4.65) -- (4,4.65) node [below]{\small 2};
       \draw [-latex] (4,4.65) -- (5,4.65);
       \draw [latex-] (-0.5,0) -- (-0.5,2.5) node [left]{\small 5};
       \draw [-latex] (-0.5,2.5) -- (-0.5,5);
       \draw [line width=0.35mm, black ] (0,0) -- (1,0) -- (1,2) -- (4,2) -- (4,0) -- (5,0) -- (5,5) -- (0,5) -- (0,0);
       \draw node at (0.5,0.3) {$\Gamma$};
       \draw node at (4.5,0.3) {$\Gamma$};
       \draw node at (2.5,3.5) {\Large$\Omega$};
       \draw node at (2.5,6.2) {$g$};
       \draw node at (2.5,4.6) {$\Gamma^N$};
       \draw [line width=1mm,-latex] (2.15,5.8) -- (2.15,5.1);
       \draw [line width=1mm,-latex] (2.5,5.8) -- (2.5,5.1);
       \draw [line width=1mm,-latex] (2.85,5.8) -- (2.85,5.1);
\end{tikzpicture}
\end{center}
\caption{Visualisation of the bridge}
\label{fig:bridge_sketch}
\end{figure}
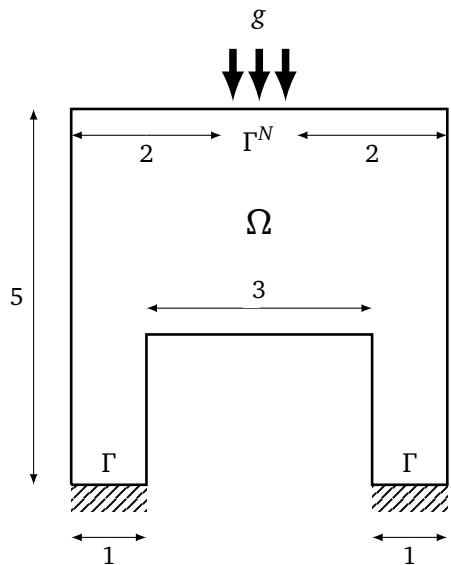

\begin{figure}[H]

\begin{subfigure}{.249\linewidth}
  \includegraphics[width=\linewidth]{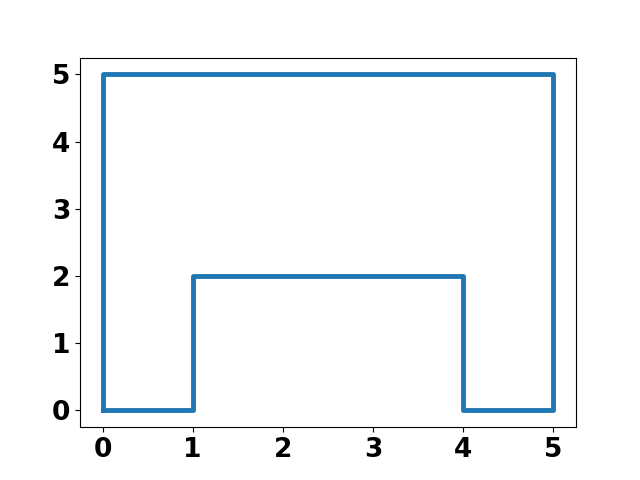}
  \caption{max-norm - It: 0}
  \label{fig:bridge_nonsmooth_0it}
\end{subfigure}\hfill 
\begin{subfigure}{.249\linewidth}
  \includegraphics[width=\linewidth]{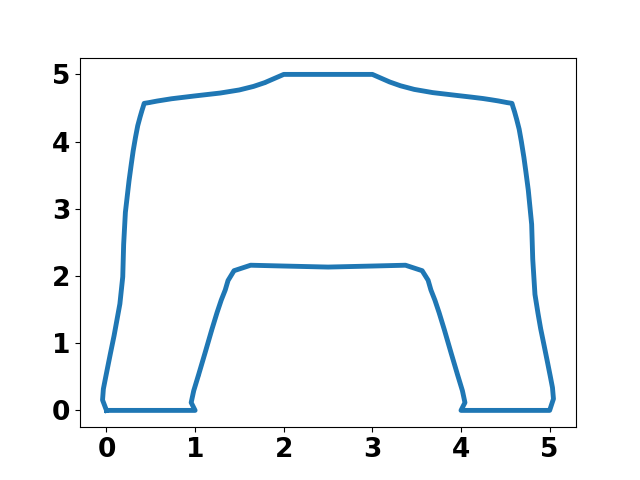}
  \caption{max-norm - It: 20}
  \label{fig:bridge_nonsmooth_20it}
\end{subfigure}\hfill 
\begin{subfigure}{.249\linewidth}
  \includegraphics[width=\linewidth]{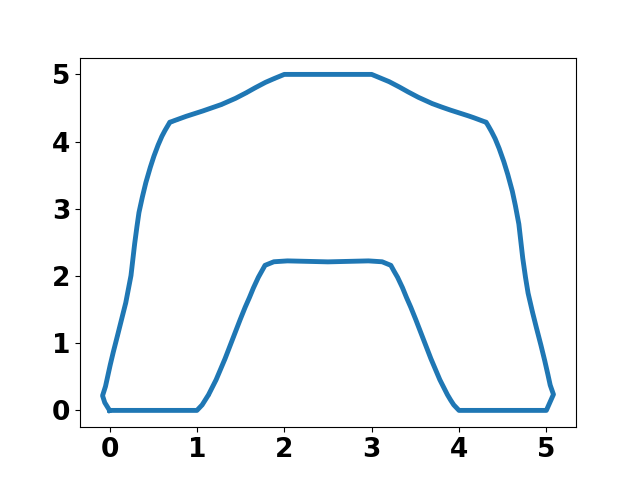}
  \caption{max-norm - It: 40}
  \label{fig:bridge_nonsmooth_40it}
\end{subfigure}\hfill 
\begin{subfigure}{.249\linewidth}
  \includegraphics[width=\linewidth]{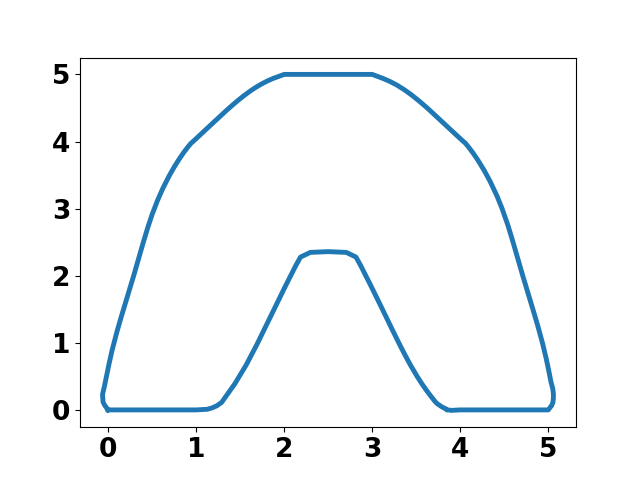}
  \caption{max-norm - It: 80}
  \label{fig:bridge_nonsmooth_80it}
\end{subfigure}

\medskip 

\begin{subfigure}{.249\linewidth}
  \includegraphics[width=\linewidth]{bridge_nonsmooth_0.png}
  \caption{$p$-norm - It: 0}
  \label{fig:bridge_smooth_0it}
\end{subfigure}\hfill 
\begin{subfigure}{.249\linewidth}
  \includegraphics[width=\linewidth]{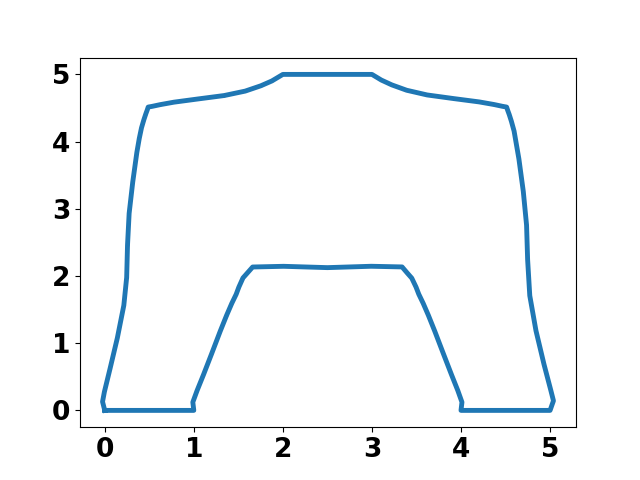}
  \caption{$p$-norm - It: 20}
  \label{fig:bridge_smooth_20it}
\end{subfigure}\hfill 
\begin{subfigure}{.249\linewidth}
  \includegraphics[width=\linewidth]{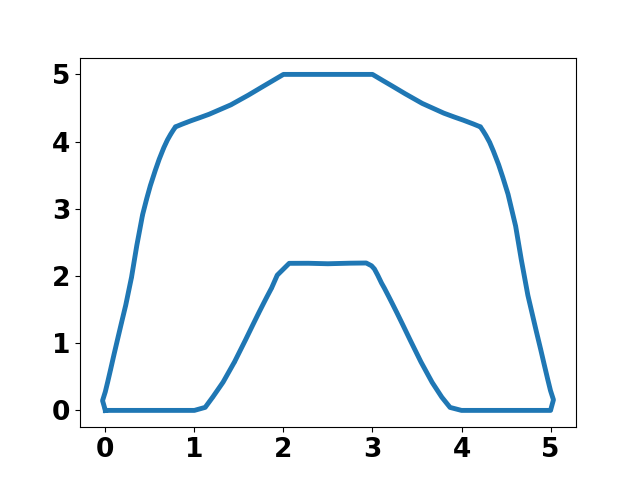}
  \caption{$p$-norm - It: 40}
  \label{fig:bridge_smooth_40it}
\end{subfigure}\hfill 
\begin{subfigure}{.249\linewidth}
  \includegraphics[width=\linewidth]{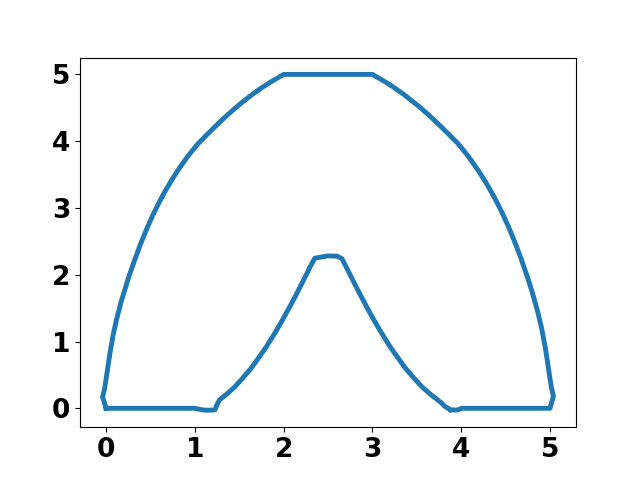}
  \caption{$p$-norm - It: 80}
  \label{fig:bridge_smooth_80it}
\end{subfigure}

\caption{Deformation of the bridge}
\label{fig:bridge_deformation}
\end{figure}

\begin{figure}[H]

\begin{subfigure}{.32\linewidth}
  \includegraphics[width=\linewidth]{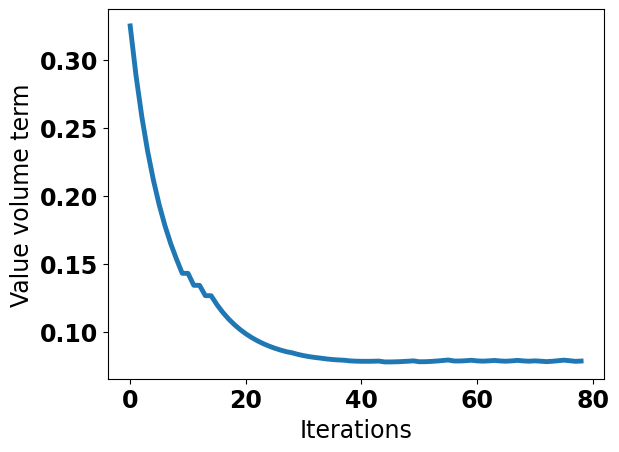}
  \caption{max-norm: volume cost}
  \label{fig:nonsmooth_bridge_volcost}
\end{subfigure}\hfill 
\begin{subfigure}{.32\linewidth}
  \includegraphics[width=\linewidth]{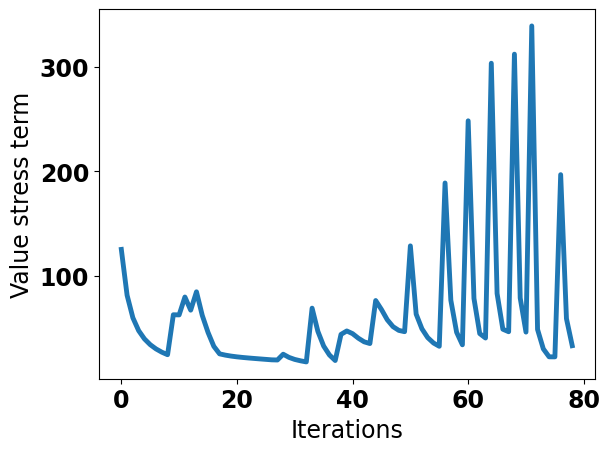}
  \caption{max-norm: stress cost}
  \label{fig:nonsmooth_bridge_stresscost}
\end{subfigure}\hfill 
\begin{subfigure}{.32\linewidth}
  \includegraphics[width=\linewidth]{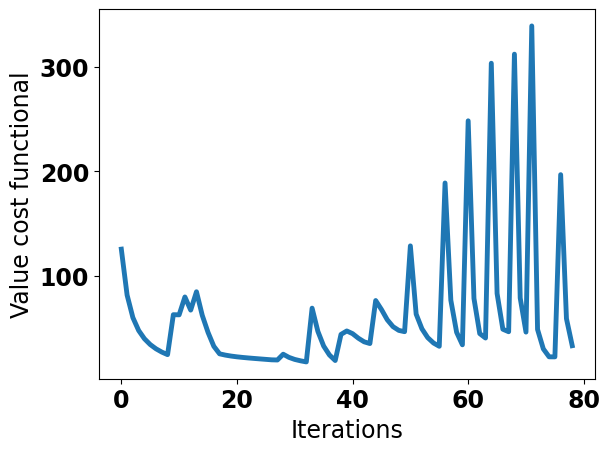}
  \caption{max-norm: total cost}
  \label{fig:nonsmooth_bridge_fullcost}
\end{subfigure}

\medskip 

\begin{subfigure}{.32\linewidth}
  \includegraphics[width=\linewidth]{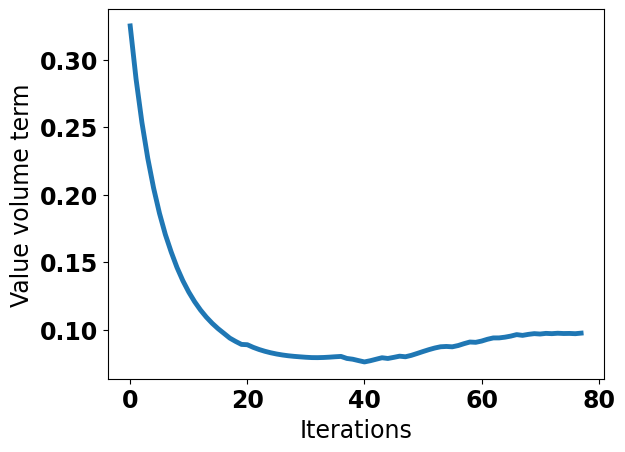}
  \caption{$p$-norm: volume cost}
  \label{fig:smooth_bridge_volcost}
\end{subfigure}\hfill 
\begin{subfigure}{.32\linewidth}
  \includegraphics[width=\linewidth]{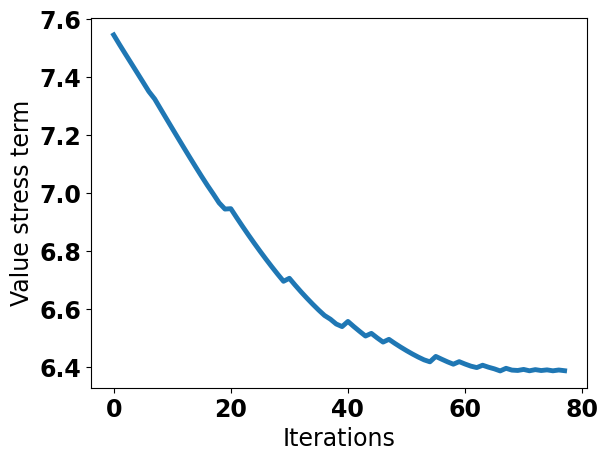}
  \caption{$p$-norm: stress cost}
  \label{fig:smooth_bridge_stresscost}
\end{subfigure}\hfill 
\begin{subfigure}{.32\linewidth}
  \includegraphics[width=\linewidth]{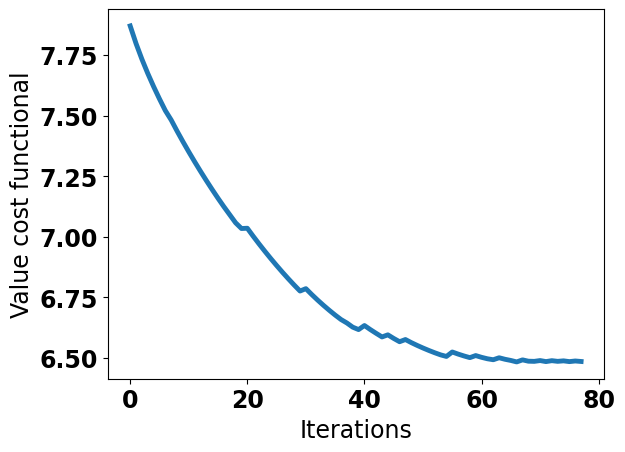}
  \caption{$p$-norm: total cost}
  \label{fig:smooth_bridge_fullcost}
\end{subfigure}

\caption{Cost evolution for the bridge problem}
\label{fig:bridge_cost}
\end{figure}

Again, we solved the PDE on the deformed domains after a fixed number of iterations with the identical meshsize $h^\text{min}=\min\{h^\text{min}_{p\text{-norm}}, h^\text{min}_\text{max-norm}\}$ and computed the maximal von Mises stress. The results for some iterations are listed in Table \ref{tab:bridge_comp}.

\begin{table}[H]
\begin{center}
\begin{tabular}{ |c||c|c| } 
 \hline
  & nonsmooth & smooth \\ 
 \hline\hline
 10 iterations & 92 & 86 \\ 
 \hline
  20 iterations & 52 & 76 \\ 
 \hline
  40 iterations & 214 & 593 \\ 
 \hline
   80 iterations & 964 & 1417 \\ 
 \hline
\end{tabular}
\end{center}
\caption{Comparison of the maximal von Mises stress $\sigma_M^2$}
\label{tab:bridge_comp}
\end{table}
In this example we are able to observe different behaviors of our approaches. While the $p$-norm approach yields a steady decline of the stress cost (Figure \ref{fig:smooth_bridge_stresscost}), the max-norm cost functional is vulnerable to remeshes and thus the associated curve shows some peaks (Figure \ref{fig:nonsmooth_bridge_stresscost}). Furthermore, we observe that the volume cost in the max-norm approach is monotone decreasing until it reaches a steady behavior at $\approx40$ iterations (Figure \ref{fig:nonsmooth_bridge_volcost}). Contrary, the $p$-norm approach shows a slightly faster decrease of the volume cost during the first few iterations (Figure \ref{fig:smooth_bridge_volcost}). Yet, the minimisation of the stress functional causes a minor increase of the volume cost. This can also be seen in the final shapes (Figure \ref{fig:bridge_smooth_80it}, Figure \ref{fig:bridge_nonsmooth_80it}), where the $p$-norm approach yields slightly wider bridge piers. Furthermore, the corners of the Dirichlet boundaries, i.e. the corners of the lower part of the bridge piers, attain a smoother appearance in the max-norm approach. This is also reflected by the last entry of Table \ref{tab:bridge_comp}, which shows a decrease of the maximal von Mises stress of $32\%$. In addition the averaging parameter $p$ had to be reduced compared to the previous example, since the $p$-norm approach did not manage to smoothen the interior corners for larger $p$.
\section{Conclusion and outlook}\label{sec:conclusion}
In this paper we derived a nonsmooth methodology to minimise peak stresses in the context of shape optimisation. We computed the associated semiderivatives and put them in the context of Clarke subgradients. Our numerical examples suggest that the nonsmooth approach yields faster stress minimisation compared to the regularised $p$-norm approach. Additionally, the max-norm approach does not entail the necessity to choose an appropriate averaging parameter $2\le p<\infty$. Contrary to common observations the max-norm algorithm yielded a similar runtime compared to the $p$-norm approach. This was enabled, since the maximiser of the von Mises stress functional in each iteration was unique.\newline
For future research, it would be interesting to investigate the adjoint variable corresponding to the stress functional in a very-weak sense. Furthermore, the efficient numerical treatment of these low regularity solutions could yield an improvement of our proposed method.
\subsection*{Acknowledgements}
Phillip Baumann has been funded by the Austrian Science Fund (FWF) project P 32911.

 \bibliographystyle{plain}
 \bibliography{stress_constraints}
\end{document}

%% file: slidesymbols.tex
\providecommand{\Div}{\operatorname{div}}          

\providecommand{\Det}{\operatorname{det}}                    

\providecommand{\argmin}{\operatorname*{argmin}}  
\providecommand{\Id}{\Op{Id}}                     

\providecommand{\limsup}{\operatorname*{limsup}}

\newcommand{\VN}{{\mathbf{N}}}

\newcommand{\VR}{{\mathbf{R}}}

\newcommand{\Dsf}{{\mathsf{D}}}

\providecommand{\Ch}{{\cal H}}

\newcommand*{\Op}[1]{\mathsf{#1}} 

